\documentclass[11pt]{article}
\usepackage{amsfonts,mathbbol}
\usepackage{amscd}
\usepackage{amssymb}
\usepackage{amsthm}
\usepackage{amsmath, xspace}
\usepackage{blkarray}
\usepackage{bm}
\usepackage{cancel}
\usepackage{color}
\usepackage{graphics}
\usepackage{graphicx}
\usepackage{enumitem}
\usepackage{fancyhdr}
\usepackage{mathdots}
\usepackage{mathrsfs}
\usepackage{multicol}
\usepackage{stmaryrd}
\usepackage{ytableau}
\usepackage[all]{xy}
\usepackage{dsfont}
\usepackage[plainpages,backref]{hyperref}

\usepackage{lscape}

\theoremstyle{plain}
\newtheorem{theorem}{Theorem}[section]
\newtheorem{main*}{Main Theorem}
\newtheorem{question}{Question}
\newtheorem{lemma}[theorem]{Lemma}
\newtheorem{proposition}[theorem]{Proposition}
\newtheorem{corollary}[theorem]{Corollary}
\theoremstyle{definition}
\newtheorem*{definition}{Definition}

\theoremstyle{remark}

\numberwithin{equation}{section}

\usepackage{ytableau}

\DeclareMathOperator{\Res}{Res}
\DeclareMathOperator{\Ind}{Ind}

\DeclareMathOperator{\maj}{maj}

\DeclareMathOperator{\Irr}{Irr}
\usepackage{tikz}
\usepackage{float}

\DeclareMathOperator{\SYT}{SYT}

\title{On the Positivity of Dihedral Branching Coefficients of the Symmetric and Alternating Groups}
\author{Velmurugan S\footnote{Indian Institute of Science, Bangalore. \quad E-mail: velmurugan.math@gmail.com}}
\date{\vspace{-5ex}}
\begin{document}
	\maketitle
\begin{abstract}
	We determine precisely when the branching coefficients arising from the restriction of irreducible representations of the symmetric group $S_n$  to the dihedral subgroup $D_n$ are nonzero, and we establish uniform linear lower bounds outside a finite exceptional family.
	As a consequence, we recover and substantially generalize known positivity results for cyclic subgroups $C_n \le S_n$.
	Analogous results are obtained for the alternating group $A_n$.
\end{abstract}



\section{Introduction}
A fundamental problem in the representation theory of finite groups is to describe the decomposition of an irreducible representation upon restriction to a subgroup.
More precisely, given irreducible representations $(\chi,V)$ of $G$ and $(\psi,W)$ of $H$, one seeks a combinatorial description of the multiplicity $\langle \Res^G_H \chi, \psi \rangle_H$.
In particular, one is interested in the support of these multiplicities, namely, determining when they are nonzero.

Beyond the symmetric group, these problems have been studied for various families of finite groups and their subgroups, for example, $p$-solvable groups and their cyclic $p$-subgroups~\cite{Hall_Higman_1956}, double covers of symmetric and alternating groups and their prime order cyclic subgroups~\cite{Zalesski_prime_projective_alternating_1996,Kleshchev_Zalesski_minimal_prime_elements_char_p_double_alt_2004}, finite simple groups of Lie type and centralizers of its elements~\cite{Heide_Zalesski_on_Passman_2006,Heide_Saxl_Tiep_Zalesski_2013},
	Chevalley groups and their cyclic $p$-subgroups~\cite{Zalesski_in_russian}, quasi-simple groups and their cyclic subgroups~\cite{Zalesski_quasi_cyclic_sylow_1999,Zalesski_distinct_eigenvalues_2006,Zalesski_quasi_simple_prime_2008}.
	In nearly all of these works, the focus is on cyclic subgroups of prime order.
	The corresponding problem for arbitrary cyclic subgroups of finite groups remains largely open.
	We refer the reader to the survey articles of Zalesskii for more details~\cite{survey_Zalesski_algebraic_Chevalley_2009, survey_Zalesski_eigenvalue_1_2025}.

	In this article, we study restriction problems for natural families of subgroups of the symmetric group $S_n$.
	The most natural subgroups to begin with are the point-stabilizer subgroups $S_{n-1}$.
	The branching rules for the restriction from $S_n$ to $S_{n-1}$ are well-known and have a beautiful combinatorial description in terms of Young diagrams (see~\cite[Chapter 2]{JamesKerber}).
	The next natural subgroups to consider are the Young subgroups of $S_n$ whose branching rules are governed by the Littlewood--Richardson rule (see~\cite[Chapter 2]{JamesKerber},~\cite[Chapter I]{Mac_sym}).
	Another important family of subgroups of $S_n$ are the Sylow $p$-subgroups: for these subgroups, a lot of work has been done (motivated in part by connections with the Mckay conjecture (now a theorem~\cite{Cabanes_Spath_2025})) in recent years, see~\cite{Giannelli_Law_Sylow_branching_2021,Giannelli_Volpato_hook_partitions_2024} and the references therein.
	A further family of subgroups of $S_n$ are the centralizers of its elements, which have been studied in connection with global conjugacy classes~\cite{Heide_Zalesski_on_Passman_2006,Sundaramconjecture,Sundaram_classes_S_n_2018,p2024cyclic}.
	The problem of restriction to cyclic subgroups of $S_n$ generated by an $n$-cycle has been studied extensively in the literature~\cite{Klyachko,M.Johnson,Swanson,Yang_Staroletov,vel_dec_2024,Schocker_kly}.
	
	The branching coefficients of arbitrary cyclic subgroups of $S_n$ has been given in~\cite{Jollenbeck_Schocker_cyclic_2000}; the positivity of these coefficients has been studied recently in~\cite{Inv_vectors,fpsac2024,p2024cyclic,vel_dec_2024} and completed by Staroletov~\cite{Staroletov_all_eigenvalues}.
	In contrast, much less is known about restrictions to noncyclic subgroups of $S_n$ that are not Young subgroups or Sylow-p-subgroups.
	Dihedral subgroups $D_n$ provide the smallest and most natural class of nonabelian subgroups containing an $n$-cycle of $S_n$.
	To the best of our knowledge, the restriction of irreducible representations of $S_n$ to dihedral subgroups has not previously been systematically studied.
	In this article, we focus on the dihedral subgroup $D_n$ of $S_n$. 	As a consequence, we also obtain results for the cyclic subgroup $C_n \le S_n$ generated by an $n$-cycle, generalizing several known results in the literature.

Let $r=(1~2~\dotsc~n)$ be the rotation and $s=(1~n)(2~n-1)\dotsc(\lfloor n/2\rfloor~\lceil n/2\rceil+1)$ be the reflection generating the dihedral subgroup $D_n = \langle r,s \rangle$ of $S_n$ of order $2n$.
The following are all the irreducible characters of $D_n$:
\begin{itemize}
	\item $\mathbb{1_1^1}$: the trivial character which takes both $r$ and $s$ to $1$,
	\item $\mathbb{1_1^{-1}}$: the character which takes $r$ to $1$ and $s$ to $-1$,
	\item $\mathbb{1_{-1}^1}$: the character which takes $r$ to $-1$ and $s$ to $1$ (only when $n$ is even),
	\item $\mathbb{1_{-1}^{-1}}$: the character which takes both $r$ and $s$ to $-1$ (only when $n$ is even),
	\item $\psi_j:=\Ind^{D_n}_{C_n} \delta^j$: the irreducible character of degree $2$ for $1\leq j \leq \lfloor (n-1)/2 \rfloor$.
\end{itemize}
For a partition $\lambda$ of $n$, let $d_{\pm\mathbb{1}}^{\pm \mathbb{1}}(\lambda) = \langle \Res^{S_n}_{D_n} \chi_\lambda, \mathbb{1}_{\pm \mathbb{1}}^{\pm \mathbb{1}} \rangle$ and $d_\psi^j(\lambda) = \langle \Res^{S_n}_{D_n} \chi_\lambda, \psi_j \rangle$.
Since tensoring with the sign character corresponds to conjugation of partitions, i.e. $\chi_{(1^n)}\otimes\chi_\lambda=\chi_{\lambda'}$, we have $d_{-\mathbb{1}}^{\pm \mathbb{1}}(\lambda) = d_{\mathbb{1}}^{\mathbb{1}}(\lambda') \text{ or } d_{\mathbb{1}}^{-\mathbb{1}}(\lambda')$, where $\lambda'$ is the conjugate partition of $\lambda$.
Therefore, it suffices to consider $d_{\mathbb{1}}^{\pm \mathbb{1}}(\lambda)$  and $d_\psi^j(\lambda)$.

Our main theorem completely determines the positivity of branching coefficients for the symmetric group $S_n$ to its dihedral subgroup $D_n$ (for $n\geq 11$) and provides uniform linear lower bounds outside a finite exceptional family.
\begin{theorem}
	\label{theorem:dihedral_branching_positivity}
	Let $n\geq 11$ be a positive integer and $D_n$ be the dihedral group of order $2n$.
	Then we have the following:

	\noindent For the two-dimensional irreducible characters of $D_n$:

	\begin{enumerate}
		\item $d_{\psi}^j(\lambda)=0$ if $\lambda \in \{(n),(1^n)\}$,
		\item $d_{\psi}^j(\lambda)=1$ if $\lambda \in \{(n-1,1),(2,1^{n-2})\}$.
	\end{enumerate}
	Suppose that $n$ is odd and $(n-1)/2$ is even.
	Then
	\begin{enumerate}
		\item $d_{\mathbb{1}}^{\mathbb{1}}(\lambda)=0$ if $\lambda \in \{(n-1,1), (n-2,1,1), (3,1^{n-3}), (2,1^{n-2})\}$,
		\item $d_{\mathbb{1}}^{\mathbb{1}}(\lambda)=1$ if $\lambda \in  \{(n),(1^n)\}$,
		\item $d_{\mathbb{1}}^{-\mathbb{1}}(\lambda)=0$ if $\lambda \in \{(n),(n-1,1), (n-2,2), (2^2,1^{n-4}),(2,1^{n-2}),(1^n)\}$.
		\item $\min\{d_{\mathbb{1}}^{\-\mathbb{1}}((2^2,1^{n-4})),d_{\mathbb{1}}^{-\mathbb{1}}((3,1^{n-3})),d_{\mathbb{1}}^{\mathbb{1}}((n-2,2)),d_{\mathbb{1}}^{-\mathbb{1}}((n-2,1,1))\} > \tfrac{n}{6}$.
	\end{enumerate}
	Suppose that $n$ is odd and $(n-1)/2$ is odd.
	Then
	\begin{enumerate}
		\item $d_{\mathbb{1}}^{\mathbb{1}}(\lambda)=0$ if $\lambda \in \{(n-1,1), (n-2,1,1), (2^2,1^{n-4}), (2,1^{n-2}),(1^n)\}$,
		\item $d_{\mathbb{1}}^{\mathbb{1}}(\lambda)=1$ if $\lambda = (n)$,
		\item $d_{\mathbb{1}}^{-\mathbb{1}}(\lambda)=0$ if $\lambda \in \{(n),(n-1,1), (n-2,2), (3,1^{n-3}),(2,1^{n-2})\}$,
		\item $d_{\mathbb{1}}^{-\mathbb{1}}(\lambda)=1$ if $\lambda = (1^n)$.
		\item $\min\{d_{\mathbb{1}}^{\mathbb{1}}((3,1^{n-3})),d_{\mathbb{1}}^{-\mathbb{1}}((2^2,1^{n-4})),d_{\mathbb{1}}^{\mathbb{1}}((n-2,2)),d_{\mathbb{1}}^{-\mathbb{1}}((n-2,1,1))\} > \tfrac{n}{6}$.
	\end{enumerate}
	Suppose that $n$ is even and $n/2$ is even.
	Then
	\begin{enumerate}
		\item $d_{\mathbb{1}}^{\mathbb{1}}(\lambda)=0$ if $\lambda \in \{(n-1,1), (n-2,1,1), (2,1^{n-2}),(1^n)\}$,
		\item $d_{\mathbb{1}}^{\mathbb{1}}(\lambda)=1$ if $\lambda = (n)$,
		\item $d_{\mathbb{1}}^{-\mathbb{1}}(\lambda)=0$ if $\lambda \in \{(n),(n-1,1), (n-2,2),(1^n)\}$,
		\item $d_{\mathbb{1}}^{-\mathbb{1}}(\lambda)=1$ if $\lambda = (2,1^{n-2})$.
		\item $\min\{d_{\mathbb{1}}^{\mathbb{1}}((n-2,2)),d_{\mathbb{1}}^{-\mathbb{1}}((n-2,1,1))\} > \tfrac{n}{6}$.
	\end{enumerate}
	Suppose that $n$ is even and $n/2$ is odd.
	Then
	\begin{enumerate}
		\item $d_{\mathbb{1}}^{\mathbb{1}}(\lambda)=0$ if $\lambda \in \{(n-1,1), (n-2,1,1), (1^n)\}$,
		\item $d_{\mathbb{1}}^{\mathbb{1}}(\lambda)=1$ if $\lambda \in  \{(n), (2,1^{n-2})\}$,
		\item $d_{\mathbb{1}}^{-\mathbb{1}}(\lambda)=0$ if $\lambda \in \{(n),(n-1,1), (n-2,2), (2,1^{n-2}), (1^n)\}$.
		\item $\min\{d_{\mathbb{1}}^{\mathbb{1}}((n-2,2)),d_{\mathbb{1}}^{-\mathbb{1}}((n-2,1,1))\} > \tfrac{n}{6}$.
	\end{enumerate}
	In all other cases, $d_{\mathbb{1}}^{\pm \mathbb{1}}(\lambda) > \tfrac{n}{12}$ and $d_\psi^j(\lambda)>\tfrac{n}{6}$ for all partitions $\lambda$ of $n$.
\end{theorem}
For $n<11$, there are additional exceptions which do not follow the above pattern; therefore, we have restricted to $n\geq 11$.
The linear bound $n/6$ is motivated by Theorem~\ref{theorem:kovacs_stohr} below.

Before discussing the consequences of our main theorem, let us recall some known results about the branching coefficients of $S_n$ to its cyclic subgroup $C_n$ generated by an $n$-cycle.
The first result about these coefficients was due to Klyachko~\cite{Klyachko} who related these coefficients to the free Lie modules and obtained the following positivity result.
\begin{theorem}\cite[Proposition 2]{Klyachko}
	\label{theorem:Klyachko_n_cycle}
	Let $n$ be a positive integer and $\lambda$ be a partition of $n$.
	Let $\delta$ be a faithful irreducible character of $C_n$.
	Then $\langle \Res^{S_n}_{C_n} \chi_\lambda, \delta \rangle \neq 0$ except when $\lambda$ is one of $(n)$, $(1^n)$ $(n>2)$, $(2,2)$ and $(2,2,2)$.
\end{theorem}
Klyachko proved the following which connects symmetric group $S_n$ representation to the free Lie $GL_m(\mathbb{C})$-modules $L_n$.
\begin{theorem}\cite[Corollary 1]{Klyachko}
	\label{theorem:Klyachko_Lie_module}
	Let $m$ be a positive integer and $\lambda$ be a partition of $n$ with at most $m$ parts.
	Then the number of times the $GL_m$-module $W_\lambda$ occurs in the free Lie module $L_n$ is equal to $\langle \Res^{S_n}_{C_n} \chi_\lambda, \delta \rangle$.
\end{theorem}

In order to provide the link between Theorem~\ref{theorem:Klyachko_Lie_module} to combinatorics of standard Young tableaux, we need to introduce a few notations.

For a partition $\lambda$ of $n$, let $\SYT(\lambda)$ denote the set of all standard Young tableaux of shape $\lambda$.
For a tableau $T\in \SYT(\lambda)$, let $\maj(T)$ denote the major index of $T$ which is defined as the sum of all descents of $T$.
We say that $i\in [n-1]$ is a descent of $T$ if $i+1$ appears in a lower row than $i$ in $T$.

Define $a_\lambda^r = |\{T\in \SYT(\lambda) \mid \maj(T)\equiv r \mod n\}|$ for $r=0,1,\dotsc,n-1$.
Kra\'skiewicz and Weyman~\cite{Kraskiewicz_Weyman} related these combinatorial quantities to the cyclic branching coefficients of $S_n$ to $C_n$ as follows.
\begin{theorem}\cite[Theorem 1]{Kraskiewicz_Weyman}
	\label{theorem:Kraskievich_weyman}
	Let $n$ be a positive integer and $\lambda$ be a partition of $n$.
	Then
	\begin{align}
		a_\lambda^r = \langle \Res^{S_n}_{C_n} \chi_\lambda, \delta^r \rangle.
	\end{align}
\end{theorem}

Therefore, for a partition $\lambda$ of $n$ with at most $m$ parts, we have
\begin{align}
	\langle W_\lambda,  L_n \rangle = a_\lambda^1= \langle \Res^{S_n}_{C_n} \chi_\lambda, \delta \rangle,
\end{align}
where $\langle W_\lambda,  L_n \rangle$ denotes the number of times the $GL_m$-module $W_\lambda$ occurs in the free Lie module $L_n$.

The positivity of one of the above coefficients (hence all) has been shown by several authors using different techniques~\cite{Klyachko,M.Johnson,Kovacs_Stohr,Swanson,Yang_Staroletov,vel_dec_2024}.
Kovacs and St\"ohr~\cite{Kovacs_Stohr} proved the positivity of $\langle W_\lambda,  L_n \rangle$ essentially using the Littlewood--Richardson rule.
Furthermore, he also gave a nice lower bound.
\begin{theorem}\cite{Kovacs_Stohr}\label{theorem:kovacs_stohr}
	Let $n>6$ be a positive integer, $L_n$ be a free Lie module of rank $m$ and $\lambda$ be a partition of $n$ with at most $m$ parts.
	Then $a_\lambda^1 \geq 1$ except when $\lambda$ is one of $(n)$, $(1^n)$.
	Furthermore, if $a_\lambda^1 >1$, then $a_\lambda^1 >\tfrac{n}{6}-1$.
\end{theorem}
Since $ d_{\mathbb{1}}^{\mathbb{1}}(\lambda) + d_{\mathbb{1}}^{-\mathbb{1}}(\lambda) = a_\lambda^0$, $d_{-\mathbb{1}}^{\mathbb{1}}(\lambda) + d_{-\mathbb{1}}^{-\mathbb{1}}(\lambda) = a_{\lambda}^{n/2}$ (when $n$ is even) and $d_{\psi}^j(\lambda) = a_\lambda^j = a_\lambda^{n-j}$ for $1\leq j \leq \lfloor (n-1)/2 \rfloor$, 
our Theorem~\ref{theorem:dihedral_branching_positivity} gives us the following result which is a generalization of Theorem~\ref{theorem:kovacs_stohr} and strengthens many known results~\cite{Swanson,Yang_Staroletov,Kovacs_Stohr,Klyachko,p2024cyclic,vel_dec_2024}.
\begin{corollary}
	Let $n\geq 11$ be a positive integer and $C_n$ be the cyclic subgroup of $S_n$ generated by an $n$-cycle.
	Let $\delta$ be a faithful linear character of $C_n$ and $a_\lambda^j = \langle \Res^{S_n}_{C_n} \chi_\lambda, \delta^j \rangle$.
	Then we have the following:
	
	\noindent For the non-real linear characters of $C_n$, 
	\begin{enumerate}
		\item $a_\lambda^j=0$ if $\lambda \in \{(n),(1^n)\}$,
		\item $a_\lambda^j=1$ if $\lambda \in \{(n-1,1),(2,1^{n-2})\}$,
	\end{enumerate}

	where $j \in 
	\begin{cases}
		\{1,2,\dotsc,n-1\} & \text{if $n$ is odd},\\
		\{1,2,\dotsc,n-1\}\setminus\{n/2\} & \text{if $n$ is even}.
	\end{cases}$

	\noindent For $n$ odd,
	\begin{enumerate}
		\item $a_\lambda^0=0$ if $\lambda \in \{(n-1,1),(2,1^{n-2})\}$,
		\item $a_\lambda^0=1$ if $\lambda \in \{(n),(1^n)\}$,
	\end{enumerate}

	\noindent For $n$ even,
	\begin{enumerate}
		\item $a_\lambda^0=0$ if $\lambda\in \{(n-1,1), (1^n)\}$,
		\item $a_\lambda^0=1$ if $\lambda \in \{(n), (2,1^{n-2})\}$,
		\item $a_\lambda^{n/2}=0$ if $\lambda \in \{(n),(2,1^{n-2})\}$,
		\item $a_\lambda^{n/2}=1$ if $\lambda \in \{(n-1,1),(1^n)\}$,
	\end{enumerate}
	In all other cases, $a_\lambda^j>\tfrac{n}{6}$ for all partitions $\lambda$ of $n$.
\end{corollary}
We remark that the lower bound $\tfrac{n}{6}$ is not optimal; for example,~\cite[Theorem 1.9]{Swanson} gives a better bounds for a large family of partitions.

Since dihedral subgroups sits inside $A_n$ when $n$ is odd, it is natural to ask whether analogous results hold for $A_n$.
Let $d_{\psi}^j(V) = \langle \Res^{A_n}_{D_n} \chi_V, \psi_j \rangle$ and $d_{\mathbb{1}}^{\pm \mathbb{1}}(V) = \langle \Res^{A_n}_{D_n} \chi_V, \mathbb{1}_{ \mathbb{1}}^{\pm \mathbb{1}} \rangle$, where $V$ is an irreducible representation of $A_n$.
\begin{theorem}
	\label{theorem:dihedral_branching_alternating}
		Let $n\geq 11$ be a positive integer and $D_n$ be the dihedral subgroup of $A_n$.
		Let $V$ be an irreducible representation of $A_n$ corresponding to a partition $\lambda$ of $n$.
	Then we have the following.

	\noindent For the two-dimensional characters of $D_n$:
	\begin{enumerate}
		\item $d_{\psi}^j(V)=0$ if $V=V_{(n)}$,
		\item $d_{\psi}^j(V)=1$ if $V=V_{(n-1,1)}$.
	\end{enumerate}
For the linear characters of $D_n$:
	\begin{enumerate}
		\item $d_{\mathbb{1}}^{\mathbb{1}}(V)=0$ if $V\in \{V_{(n-1,1)}, V_{(n-2,1,1)}\}$,
		\item $d_{\mathbb{1}}^{\mathbb{1}}(V)=1$ if $V = V_{(n)}$,
		\item $d_{\mathbb{1}}^{-\mathbb{1}}(V)=0$ if $V \in \{V_{(n)}, V_{(n-1,1)}, V_{(n-2,2)}\}$.
		\item $\min\{d_{\mathbb{1}}^{\mathbb{1}}((n-2,2)),d_{\mathbb{1}}^{-\mathbb{1}}((n-2,1,1))\} > \tfrac{n}{6}$.
	\end{enumerate}
	In all other cases, $d_{\mathbb{1}}^{\pm \mathbb{1}}(\lambda) > \tfrac{n}{12}$ and $d_\psi^j(\lambda)>\tfrac{n}{6}$ for all partitions $\lambda$ of $n$.
\end{theorem}

The following corollary generalizes the results of~\cite{Swanson,p2024cyclic,vel_dec_2024} for the alternating group $A_n$.
\begin{corollary}
	Let $n\geq 11$ be an odd positive integer and $C_n$ be the cyclic subgroup of $A_n$ generated by an $n$-cycle.
	Let $\delta$ be a faithful linear character of $C_n$ and $a_V^j = \langle \Res^{A_n}_{C_n} \chi_V, \delta^j \rangle$.
	Then we have the following:
	\begin{enumerate}
		\item $a_V^j=0$ if $V=V_{(n)}$, where $j \in 
		\{1,2,\dotsc,n-1\}$.
		\item $a_V^j=1$ if $V = V_{(n-1,1)}$,
	where $j \in 
		\{1,2,\dotsc,n-1\}$.
		\item $a_V^0=0$ if $V \in \{V_{(n-1,1)}, V_{(2,1^{n-2})}\}$,
		\item $a_V^0=1$ if $V \in \{V_{(n)}, V_{(1^n)}\}$,
	\end{enumerate}
	In all other cases, $a_V^j>\tfrac{n}{6}$ for all irreducible representations $V$ of $A_n$.
\end{corollary}
Let us give the immediate consequences of our main theorems.
We shall start with the Hecke algebras corresponding to the characters of $D_n$.
For the definitions of Hecke algebras and their representations, we refer the reader to~\cite{Curtis_Reiner_1981_book}.
\begin{corollary}
	Let $n\geq 11$ be a positive integer and $D_n$ be the dihedral group of order $2n$.
	Let $\mathcal{H}:=\mathcal{H}(S_n,D_n,\chi)$ be the Hecke algebra corresponding to the character $\chi$ of $D_n$.
	Then the number of irreducible representations of $\mathcal{H}$ is equal to the number of partitions of $n$ minus $k$, where $k\leq 5$ can be read off explicitly from Theorem~\ref{theorem:dihedral_branching_positivity}. 
\end{corollary}
As we have seen that $d_{\mathbb{1}}^{\mathbb{1}}(\lambda) + d_{\mathbb{1}}^{-\mathbb{1}}(\lambda) = a_\lambda^0$, $d_{-\mathbb{1}}^{\mathbb{1}}(\lambda) + d_{-\mathbb{1}}^{-\mathbb{1}}(\lambda) = a_{\lambda}^{n/2}$ (when $n$ is even), it is natural to ask the following question.
\begin{question}
	Is there a combinatorial interpretation for the dihedral branching coefficients of $S_n$ to the linear characters of $D_n$?
	More precisely, is there a statistic $\mathrm{stat}$ on $\SYT(\lambda)$ which refines the major index, i.e., on the set $\{T\in \SYT(\lambda) \mid \maj(T)\equiv 0 \mod n\}$ (and similarly on the set $\{T\in \SYT(\lambda) \mid \maj(T)\equiv n/2 \mod n\}$ when $n$ is even$)$ such that this statistic recovers the dihedral branching coefficients for the linear characters of $D_n$?
\end{question}
Similarly, we have the following question for the Hecke algebras.
\begin{question}
	Give a combinatorial interpretation for the dimensions of the Hecke algebras corresponding to the characters of $D_n$?
\end{question}

\subsection*{Supporting Code}
Some of the steps in the proofs involve direct calculations using the \texttt{SageMath}~\cite{sagemath}.
Code for carrying out these calculations can be downloaded from:\\
\href{https://github.com/velmurugan1066/Dihedral-restriction-coefficients/blob/main/index.html}{https://github.com/velmurugan1066/Dihedral-restriction-coefficients}

\section{Notation and Preliminaries}\label{section:Notation and preliminaries}
Let $G$ be a finite group and $H$ be a subgroup of $G$.
Let $\Irr(G)$ denote the set of all complex irreducible characters of $G$.
For  $\Psi\in \Irr(H)$, let $\Irr(G|\Psi)=\{\chi\in \Irr(G)\mid \langle \Res^G_H \chi,\Psi\rangle_H\neq 0\}$.
We denote the trivial character of $H$ by $\mathbb{1}_H$.
Let $(\chi,V)$ and $(\Psi,W)$ be representations of $G$.
If $(\chi,V)$ is a subrepresentation of $(\Psi,W)$, then we denote this by $\chi\leq\Psi$ or $V\leq W.$

For a composition $\alpha$ of $n$, we define a natural permutation $w_\alpha$ of $[n]$ as follows:
\begin{displaymath}
    w_\alpha=(1~2~\dotsc~\alpha_1)(\alpha_1+1~\alpha_1+2~\dotsc~\alpha_1+\alpha_2)\dotsc(\sum_{i=1}^{k-1}\alpha_i+1 ~\sum_{i=1}^{k-1}\alpha_i+2~\dotsc~\sum_{i=1}^{k}\alpha_i=n),
\end{displaymath}
where $k$ is the length of $\alpha$.
Let $C_\alpha$ denote the cyclic subgroup generated by $w_\alpha$. When $\alpha=(n)$, we write $C_n$ instead of $C_{(n)}$.
Let $D_n$ denote the dihedral subgroup of $S_n$ generated by the rotation $w_n=r=(1~2~\dotsc~n)$ and the reflection $
s=(1~n)(2~n-1)\dotsc(\lfloor n/2\rfloor~\lceil n/2\rceil+1)$ with cycle type $(2^{\lfloor n/2\rfloor},1^{n-2\lfloor n/2\rfloor})$,
\begin{align*}
D_n=\langle (1~2~\dotsc~n),(1~n)(2~n-1)\dotsc(\lfloor n/2\rfloor~\lceil n/2\rceil+1)\rangle.
\end{align*}

Let us recall the construction of wreath products and their embedding into symmetric groups (see~\cite[Chapter 4]{JamesKerber} for more details).
For subgroups $H\leq S_n$ and $K\in S_m$, we have a natural isomorphism (see~\cite[4.1.18]{JamesKerber}) from $H\wr K$ into $S_{mn}$ as follows:
\begin{displaymath}
    \Gamma: H\wr K\rightarrow S_{mn}: (f,\pi)\mapsto \left(\begin{array}{@{}*{20}{c@{\,}}} (j-1)m+i
 \\ (\pi(j)-1)m+f(\pi(j))(i)
  \\
\end{array}\right)_{1\leq i\leq m,~ 1\leq j\leq n}.
\end{displaymath}
The following lemmas will play a central role in our proofs.
\begin{lemma}\label{lemma:mn-cycle}
Let $C_n\leq S_n$ and $C_m\leq S_m$.
Then $\Gamma(C_m\wr C_n)$ contains an $mn$-cycle. 
\end{lemma}
\begin{proof}
    Let $\sigma_{mn}=( ((123\dots m),e,e\dots,e);(12\dots n) )$ be an element of $C_m\wr C_n$.
Then we have
    \begin{displaymath}
    \tau_{mn}:=\Gamma(\sigma_{mn}) =(1~m+1~2m+1~\dots(n-1)m+1~2~m+2~2m+2~\dots(n-1)m+2~\dots~m~2m~3m\dots nm)    .
    \end{displaymath}    
\end{proof}

\begin{lemma}\label{lemma:mn_dihedral}
	Let $D_n\leq S_n$ and $D_m\leq S_m$.
	Then $\Gamma(D_m\wr D_n)\leq S_{mn}$ contains a subgroup $E_{mn}$ that is conjugate to $D_{mn}$.
\end{lemma}
\begin{proof}
	From the previous lemma, we know that $\Gamma(D_m\wr D_n)$ contains the $mn$-cycle $\tau_{mn}$ ($=\Gamma(\sigma_{mn})$).
	Recall that every $mn$-cycle is associated with a unique dihedral subgroup of $S_{mn}$ of order $2mn$ and is conjugate to $D_{mn}$.
	Now it suffices to show that there exists an element $t\in D_m\wr D_n$ of order $2$ such that $\Gamma(t)\Gamma(\sigma_{mn})\Gamma(t^{-1})=\Gamma(\sigma_{mn})^{-1}$
	(equivalently, $t \sigma_{mn} t^{-1}=\sigma_{mn}^{-1}$).
	Let $s_k=(1~n)(2~n-1)(3~n-2)~\dotsc$ be an element of $D_k$ of order $2$.
	We claim that $t=((s_m,s_m,\dotsc,s_m);s_n)$ is an element of $D_m\wr D_n$ of order $2$ and satisfies the required condition.
	To see this, we compute
	\begin{align*}
		t\sigma_{mn}t^{-1}&=((s_m,s_m,\dotsc,s_m);s_n)( ((123\dots m),e,e\dots,e);(12\dots n) )((s_m,s_m,\dotsc,s_m);s_n)^{-1}\\
		&=((s_m,s_m,\dotsc,s_m);s_n)( ((123\dots m),e,e\dots,e);(12\dots n) )((s_m,s_m,\dotsc,s_m);s_n)\\
		&=((s_m,s_m,\dotsc,s_m);s_n) ( ((123\dots m)s_m,s_m,s_m,\dots,s_m)  ;(1~2~\dotsc~n) s_n)\\
		&=((s_m,s_m,\dotsc,s_m) (s_m,s_m\dots,s_m,(123\dots m) s_m)  ; s_n (1~2~\dotsc~n) s_n)\\
		&=((e,\dots,e,s_m (123\dots m)s_m)  ; (1~n~n-1~\dots~2) )\\
		&=((e,\dots,e,(1~m~m-1~\dots~2))  ; (1~n~n-1~\dots~2) )\\
		&=\sigma_{mn}^{-1}.
	\end{align*}

	This completes the proof.
\end{proof}

\begin{lemma}
	\label{lemma:dihedral_young_subgroup_intersection}
	Let $I$ be a Young subgroup of $S_n$, where $n$ is an odd prime, with type $\mu$ such that $l(\mu)\geq 3$.
	There exists a $\sigma\in S_n$ such that $\sigma I \sigma^{-1} \cap D_n={e}$.
\end{lemma}
\begin{proof}
	If $1$ is a part of $\mu$ and say $k$ is fixed by $I$, then we may choose a permutation $\sigma\in S_n$ as follows:
	\begin{itemize}
		\item Firstly, map $k$ to $1$.
		\item Take two elements $l,f$ (different from $k$) belong to two different orbits of $I$.
		Map $l,f$ to $2,n$ respectively.
		\item Extends this to a bijection in any manner the reader wish.
	\end{itemize}
	Therefore, we have $J:=\sigma I \sigma^{-1}$ fixes $1$ and $2,n$ lie in different orbits of $J$.
	We see that if there is a reflection $s$ lies in $J\cap D_n$, then $s$ must fix $1$.
	Therefore, $s$ must swap $2$ and $n$, which is absurd.

	Otherwise, each part of $\mu$ is at least $2$.
	Therefore, $n\geq 7$. Let $A$ be one of the orbits of $I$ with at least $3$ elements.
	Let $\{a,b,c\}\subset A,\{d,e\}\subset D,\{f,g\}\subset F$, where $A,D,F$ are distinct orbits.
	Choose a permutation as follows:

	\begin{itemize}
		\item Firstly, map $a,b,c$ to $1,3,4$ respectively.
		\item Secondly, map $d,e$ to $n,5$ respectively.
		\item Thirdly, map $f,g$ to $n-1,2$ respectively.
		\item If there are $t$ many elements in $F$ which are not mapped, map them bijectively with the elements of $\{n-2,n-3,\dotsc,n-t-1\}$.
		\item Now extend this to a bijection in any manner the reader wish.
	\end{itemize}
	We claim that $J:=\sigma I \sigma^{-1}$ does not contain any elements of $D_n$.
	Clearly, $J$ is proper therefore it does not contain any $n$-cycles.
	If there exists a reflection $s$ in $J\cap D_n$, then $s$ cannot fix $1$ (resp. $2$).
	Therefore, $1$ must be swapped to some element in $A$ (say $h$).
	Note that no element of the set $\{1,2,\dotsc,h$\} belong to the orbit $F$ except $2$.
	At the same time the set $\{1,2,\dotsc,h\}$ is fixed by $s$.
	This implies that $s$ must fix $2$.
	A contradiction.
\end{proof}

The following lemma is repeatedly used in the proof of many lemmas.
\begin{lemma}\label{lemma:dominant_partition}
	Let $\lambda,\alpha$ be partitions of $n$ and $m$ respectively such that $\alpha\subset\lambda$.
	Suppose that $k$ is a positive integer such that $k+m\leq n$.
	Then there exists a partition $\beta\vdash k$ with the following property:
	There exists a partition $\eta\supset \alpha$, $\eta\subset \lambda$, $|\eta|=k+m$, such that $s_{\eta/\alpha}\geq s_\beta$.
	Moreover, if $\gamma\vdash k$ also satisfies this property, then $\beta\geq \gamma$ in the dominance order.
\end{lemma}
\begin{proof}
	Let us construct a partition $\beta$ of $k$ and fill the Young diagram of $\lambda/\alpha$ as follows:
	$\beta_1=\min(\text{the number of columns of } \lambda/\alpha, k)$.
	Fill $\beta_1$ many $1's$ in the first cell of each column starting from the top and rowwise left to right.
	Next,
	$\beta_2=\min(\text{the number of columns of } \lambda/\alpha \text{ with at least } 2 \text{ cells}, k-\beta_1)$.
	Fill $\beta_2$ many $2's$ in the second cell of each column of $\lambda/\alpha$ starting from the top and rowwise left to right.
	We may continue this process until we fill $\beta_k$ ($l(\beta)=k$) many $k's$ in the $k$-th cell of each column of $\lambda/\alpha$ starting from the top and rowwise.
	For example, if $\lambda=(6,4,3,2)$, $\alpha=(3,2)$ and $k=7$, then the filling of $\lambda/\alpha$ is as follows: 
\ytableausetup{boxsize=1.5em}
  $$T_{\lambda\alpha} = \begin{ytableau}
    *(red)  &*(red) &*(red) & 1 & 1 & 1 \\
    *(red)&*(red)  & 1 & 2\\
    1 & 1 &  & \\
     \\
  \end{ytableau},$$
The partition $\eta$ can be obtained by taking the shape of the filled cells in $T_{\lambda\alpha}$ and the cells of $\alpha$.
Since the filling is semistandard and the reverse row reading word is a lattice permutation, by the Littlewood--Richardson rule, we have $s_{\eta/\alpha}\geq s_\beta$.
In the above example, $\eta=(6,4,2)$ and $\beta=(6,1)$.
By construction, $\beta$ is maximal in dominance order.
\end{proof}

\begin{definition}[Dominant partition]\label{definition:dominant_partition}
	For partitions $\lambda\supset\alpha$ of $n$ and $m$ respectively, we call the partition $\beta$ of $k$ defined above the \emph{dominant partition} of $k$ with respect to $\lambda/\alpha$.
\end{definition}

We refer the reader to~\cite{Mac_sym} for the background on symmetric functions, to~\cite{JamesKerber} for the background on the representation theory of symmetric groups and of wreath products.
We use standard notation and results from these references.

\section{Propositions for the Proof of the Main Theorems}\label{section:Ingredients for the proof of the main theorems}
	We have the first key proposition as follows which allows us to use induction in our proofs.
    \begin{proposition}
		\label{proposition:Giannelli}
		Let $n$ be a positive odd integer with $n=mp$, where $p$ is a prime and $m>1$. Let $\mu^1, \dots, \mu^p$ be partitions of $m$, with at least three of them are pairwise distinct.
		Let $\lambda\vdash n$ satisfy
		 $$\Res^{S_n}_{ S_m^{\times p}}\chi_\lambda\geq \chi_{\mu^1} \times\dots \times \chi_{\mu^p} := \Psi.$$
        Then
		 $$\Ind_{\langle w_n^p \rangle}^{D_n} \Res_{\langle w_n^p \rangle}^{ S_m^{\times p}} \Psi \leq  \Res_{D_n}^{ S_n}\chi_\lambda.$$
	\end{proposition}
    \begin{proof}
The proof proceeds via a sequence of restriction--induction steps, summarized in the following diagrams. The nodes of the first diagram consist of subgroups of $ S_n$ and those of the second diagram consist of irreducible representations of the corresponding subgroups in the first diagram.

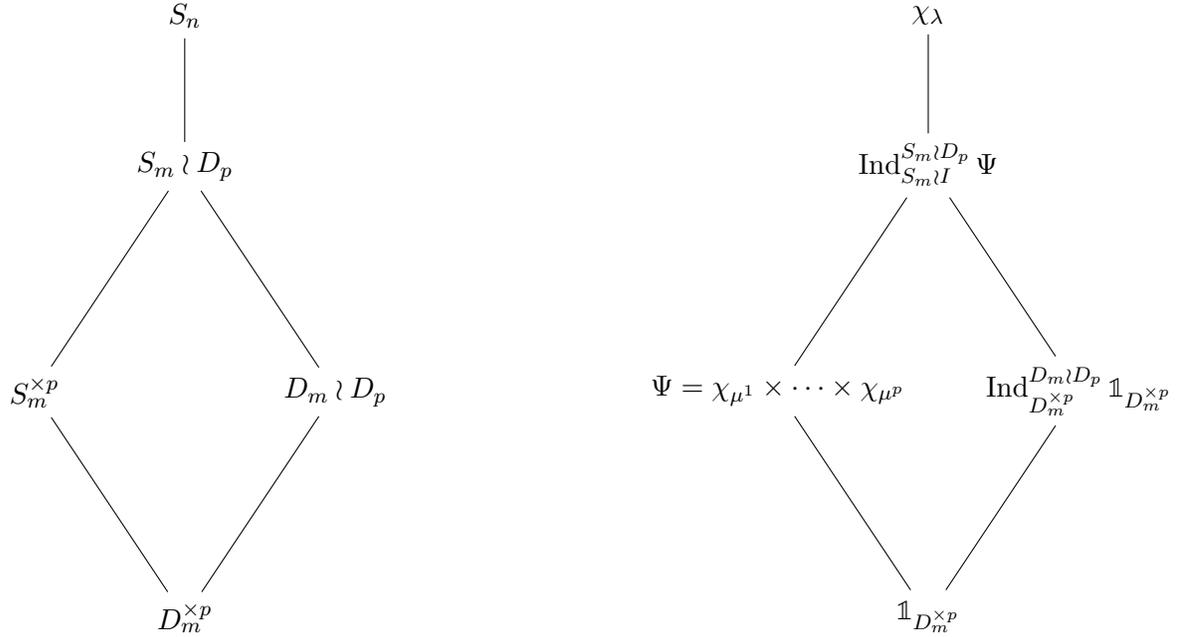
\begin{center}
\begin{tikzpicture}
    \node (Sn) at (0,0) {$S_n$};
    \node (SmSp) at (0,-2) {$S_m\wr D_p$};
    \node (Smp) at (-2,-5) {$S_m^{\times p}$};
    \node (Dmp) at (0,-8) {$D_m^{\times p}$};
    \node (DmDp) at (2,-5) {$D_m\wr D_p$};

    \draw (Sn) -- (SmSp);
    \draw (SmSp) -- (Smp);
    \draw (Smp) -- (Dmp);
    \draw (SmSp) -- (DmDp);
    \draw (DmDp) -- (Dmp);
\end{tikzpicture}
\hspace{3cm}
\begin{tikzpicture}
    \node (Sn) at (0,0) {$\chi_\lambda$};
    \node (SmSp) at (0,-2) {$\Ind_{S_m \wr I}^{S_m\wr D_p} \Psi$};
    \node (Smp) at (-2,-5) {$\Psi=\chi_{\mu^1}\times \dotsb\times \chi_{\mu^p}$};
    \node (Dmp) at (0,-8) {$\mathbb{1}_{D_m^{\times p}}$};
    \node (DmDp) at (2,-5) {$\Ind_{D_m^{\times p}}^{D_m\wr D_p} \mathbb{1}_{D_m^{\times p}}$};

    \draw (Sn) -- (SmSp);
    \draw (SmSp) -- (Smp);
    \draw (Smp) -- (Dmp);
    \draw (SmSp) -- (DmDp);
    \draw (DmDp) -- (Dmp);
\end{tikzpicture}
\end{center}

    Set $\Psi=\chi_{\mu^1}\times \dotsb \times \chi_{\mu^p}$.
	By abuse of notation, we use $D_p$ to denote the subgroup of $\Gamma(S_m \wr D_p)$ which permutes the $p$ copies of $S_m$ in $S_m^{\times p}$.
	Let $I_{S_m\wr D_p}(\Psi) = \{ g \in S_m\wr D_p \mid \Psi^g = \Psi \}$, where $\Psi^g$ is defined as $\Psi^g(x) = \Psi(g^{-1}xg)$ for $x \in S_m^{\times p}$.
	Since $p$ is prime and at least three of the partitions $\mu^1, \dots, \mu^p$ are mutually distinct, no $p$-cycles of $D_p$ belongs to $I_{S_m\wr D_p}(\Psi)$.
	We may permute (if necessary) the partitions $\mu^1, \dots, \mu^p$ using Lemma~\ref{lemma:dihedral_young_subgroup_intersection} such that no reflection of $D_p$ belongs to $I_{S_m\wr D_p}(\Psi)$.
	Therefore, $I_{S_m\wr D_p}(\Psi) = S_m^{\times p}$.
	Since the inertia group $I_{S_m\wr D_p}(\Psi)= S_m^{\times p}$ of $\Psi$ is the base group $S_m^{\times p}$, Clifford theory implies that $\Ind_{S_m^{\times p}}^{S_m\wr D_p} \Psi$ is irreducible.
	Since $\Res^{S_n}_{S_m^{\times p}}\chi_\lambda \geq \Psi$, by Frobenius reciprocity, we have $\Res^{S_n}_{S_m\wr D_p} \chi_\lambda \geq \Ind_{S_m^{\times p}}^{S_m\wr D_p} \Psi$.

	Applying Mackey's restriction formula and taking the contribution of the trivial double coset we obtain
	\begin{align*}
		\Res^{S_n}_{D_m\wr D_p} \chi_\lambda &\geq \Ind^{D_m\wr D_p}_{D_m^{\times p}} \Res^{S_m^{\times p}}_{D_m^{\times p}} \Psi.
	\end{align*}
	Now once again using the Mackey's restriction formula, we have
	\begin{align*}
		\Res_{E_n}^{D_m\wr D_p} \Ind^{D_m\wr D_p}_{D_m^{\times p}} \Res^{S_m^{\times p}}_{D_m^{\times p}} \Psi
		&=\sum\limits_{g\in D_m^{\times p} \setminus D_m\wr D_p / E_n}  \Ind^{E_n}_{{(D_m^{\times p})}^g\cap E_n} \Res^{{(D_m^{\times p})}^g}_{{(D_m^{\times p})}^g\cap E_n} {\Res^{S_m^{\times p}}_{D_m^{\times p}} \Psi}^g\\
		&\geq \Ind^{E_n}_{D_m^{\times p}\cap E_n} \Res^{D_m^{\times p}}_{D_m^{\times p}\cap E_n} \Res^{S_m^{\times p}}_{D_m^{\times p}} \Psi\\
		&= \Ind^{E_n}_{\langle \tau_n^p \rangle} \Res^{S_m^{\times p}}_{\langle \tau_n^p \rangle} \Psi,
	\end{align*}
where $E_n$ is a subgroup of $D_m\wr D_p$ conjugate to $D_n$ from Lemma~\ref{lemma:mn_dihedral} and $\tau_n^p=(1~2~\dotsc~m)(m+1~m+2~\dotsc~2m)\dotsc((p-1)m+1~(p-1)m+2~\dotsc~pm)$.

Hence, we have
	\begin{align*}
		\Res^{S_n}_{E_n} \chi_\lambda &\geq \Ind^{E_n}_{\langle \tau_n^p \rangle} \Res^{S_m^{\times p}}_{\langle \tau_n^p \rangle} \Psi.
	\end{align*}
Since $E_n$ and $D_n$ are conjugate subgroups of $S_n$, the proof follows.
\end{proof}

\begin{proposition}
	\label{proposition:even_Giannelli}
	Let $n$ be a positive even integer with $n=2m$ and $m>1$. Let $\mu^1, \mu^2$ be distinct partitions of $m$.
	Let $\lambda\vdash n$ such that $\chi_{\mu^1} \times \chi_{\mu^2}\leq\Res_{ S_m^{\times 2}}^{ S_n}\chi_\lambda$.	
    Then $\Ind_{\langle w_n^2, s \rangle}^{D_n} \Res_{\langle w_n^2, s \rangle}^{ S_m^{\times 2}} \chi_{\mu^1} \times \chi_{\mu^2} \leq \Res_{D_n}^{ S_n}\chi_\lambda$, where $s$ is some reflection in $D_n$.
\end{proposition}
\begin{proof}
 	The argument is identical to the argument of Proposition~\ref{proposition:Giannelli}, replacing  $D_p$ by $D_2$, and noting that $E_n \cap D_m^{\times 2} = \langle \tau_n^2, s \rangle$, where $s=(1)(2~m)(3~m-1)\dots (m+1~2m)(m+2~2m-1)\dots$.
\end{proof}
The following propositions will resemble the proof of Theorem~\ref{theorem:dihedral_branching_positivity} in a specific case.
\begin{proposition}
	\label{proposition:two_row_rectangular}
	Let $n=2p$ be a positive even integer with $p\geq 41$ is a prime and $\lambda=(p,p)$.
	Then $\langle \Res^{S_n}_{ D_n} \chi_\lambda, \psi \rangle > n$ for all irreducible characters $\psi$ of $D_n$.
\end{proposition}
\begin{proof}
	Let $\psi$ be an irreducible character of $E_n$ (conjugate to $D_n$ from Lemma~\ref{lemma:mn_dihedral}).
	We may start with the case when $\psi$ restricts to the trivial character of $\langle \tau_n^p \rangle$.
	Let $\mu^1=\mu^2=\dots=\mu^{i-1}=(1,1)$ and $\mu^{i}=\mu^{i+1}=\dotsc= \widehat{\mu^{j}	}= \mu^{j+1}=\dotsc=\mu^{p}=(2)$ for some $3\leq i<j\leq \tfrac{p-1}{2}$ and $\widehat{\mu^{j}} = (1,1)$.
	Then $\chi_i^j:=\chi_{\mu^1} \times\dots \times \chi_{\mu^p}$ is an irreducible character of $ S_2^{\times p}$.
	Note that $\Res_{\langle \tau_n^p \rangle}^{ S_2^{\times p}} \chi_i^j$ is the trivial character of $\langle \tau_n^p \rangle$ for all $3\leq i<j\leq \tfrac{p-1}{2}$ with $i$ odd.
	Therefore, from the proof of Proposition~\ref{proposition:Giannelli}, we have
	\begin{align*}
		\Res^{S_n}_{E_n} \chi_\lambda &\geq \Ind_{\langle \tau_n^p \rangle}^{E_n} \Res_{\langle \tau_n^p \rangle}^{ S_2^{\times p}} \chi_i^j = \Ind_{\langle \tau_n^p \rangle}^{E_n} \mathbb{1}_{\langle \tau_n^p \rangle}\geq \psi.
	\end{align*}
	We can see that $(\chi_i^j)^{g}=\chi_{i'}^{j'}$ for some $3\leq i'<j'\leq \tfrac{p-1}{2}$ with $i'$ odd and $g\in D_2\wr D_p$ if and only if $g \in D_2^{\times p}$.
	Hence, $\Ind_{S_2^{\times p}}^{D_2\wr D_p} \chi_i^j$ are mutually inequivalent irreducible characters of $D_2\wr D_p$ for distinct choices of $3\leq i<j\leq \tfrac{p-1}{2}$ with $i$ odd.
	Since the number of choices of $3\leq i<j\leq \tfrac{p-1}{2}$ with $i$ odd is equal to $\sum\limits_{i=3, ~i \text{ odd}}^{(p-1)/2} \frac{p-1}{2}-i$.
	Considering the first six terms of the summation, we have that the number of choices is at least $ 3p-18> 2p$ for all $p\geq 41$.
	This proves that
	$\langle \Res^{S_n}_{ E_n} \chi_\lambda, \psi \rangle > n$
	for the case when $\psi$ restricts to the trivial character of $\langle w_n^p \rangle$.

	An analogous construction handles the case where $\psi$ restricts to the sign character of $\langle w_n^p \rangle.$
	This completes the proof.
\end{proof}

Let us now consider one more specific case whose proof follows from the proof of Proposition~\ref{proposition:Giannelli}.
\begin{corollary}
	\label{corollary:divisible_by_4}
	Let $n=4m$ be a positive integer with $m\geq 20$ and $\lambda=(2m,2m)$.
	Suppose that $\mu^1, \mu^2, \mu^3, \mu^4$ are mutually distinct partitions of $m$ such that $\Res^{S_n}_{ S_m^{\times 4}} \chi_\lambda \geq \chi_{\mu^1} \times \chi_{\mu^2} \times \chi_{\mu^3} \times \chi_{\mu^4}$.
	Then $\langle \Res^{S_n}_{ D_n} \chi_\lambda, \psi \rangle > 0$ for every irreducible character $\psi$ of $D_n$.
\end{corollary}

\section{The lemmas}

The following lemma is fundamental in the proof of the main theorem.
\begin{lemma}
	\label{lemma:base}
Let $m\geq 11$ be a positive integer.
Then
\begin{equation}\label{eq:2m_distinct}
	\sum\limits_{\substack{\alpha,\beta\neq (m),(1^m) \\
	\alpha\neq\beta}} s_\alpha s_\beta  \geq s_\lambda,
\end{equation}
for all partitions $\lambda$ of $2m$ except when $\lambda$ is equal to one of $(2m),(2m-1,1),(2m-2,1,1),(2m-2,2), \linebreak (m,m),(3,1^{2m-3}), (2^2,1^{2m-4}),(2,1^{2m-2}),(1^{2m})$.
	
	Moreover, for a partition $\lambda$ of $2m$ which is not one of the above exceptions and $\lambda_1,\lambda_1'\leq 2m-5$, we have distinct partitions $\alpha,\beta$ such that $s_\alpha s_\beta\geq s_\lambda$ and $\alpha_1,\alpha_1',\beta_1,\beta_1'\leq m-2$.
\end{lemma}
\begin{proof}
The LHS is invariant under the standard involution $\omega$, therefore $s_\lambda$ satisfies Eq~\eqref{eq:2m_distinct}, if and only if $s_{\lambda'}$ satisfies the equation.
Thus, it suffices to consider partitions $\lambda$ with $\lambda_1\geq \lambda_1'$,
since $s_\lambda$ satisfies \eqref{eq:2m_distinct} if and only if $s_{\lambda'}$ does.
We may verify the lemma directly when $\lambda\vdash 2m$ and $\lambda_1> 2m-5$ or $\lambda_1'> 2m-5$.
Now the former statement follows from the latter statement of the lemma.
Let us start the proof of the latter statement of the lemma.
Therefore, let $\lambda\vdash 2m$ with $\lambda_1,\lambda_1'\leq 2m-5$ and $\lambda_1\geq \lambda_1'$.
Define $A_m=\{(m),(m-1,1),(2,1^{m-2}),(1^m)\}$.
We shall give a proof using case-by-case analysis.
\subsubsection*{Case I. $\lambda$ is a hook partition or a two row partition}

If $\lambda$ is a hook partition, then we may choose $\alpha=(m-2,1,1)$ and $\beta=(\lambda_1-m+2,1^{2m-\lambda_1-2})$.

\noindent Suppose that $\lambda$ is a two row partition.
If $\lambda=(m,m)$, then it is easy to see that $c_{\alpha\beta}^{(m,m)}>0$ implies that $\alpha=\beta$.
Otherwise, $\lambda\neq (m,m)$ and we may choose $\alpha=(m-2,2)$ and $\beta=(\lambda_1-m+2,2m-\lambda_1-2)$.

From now on, we may assume that $\lambda$ is neither a hook partition nor a two row partition.

\subsubsection*{Case II. $\lambda_1\geq \lambda_1'$ and $\lambda_1\geq m$.}
Let us replace $\alpha$ with $(m-2,1,1)$ and $\beta$ with the dominant partition of $m$ contained in $\lambda/\alpha$.
If $\beta\notin A_m$ and $\alpha\neq \beta$, then we are done.

Otherwise, suppose that $\beta=(m)$.
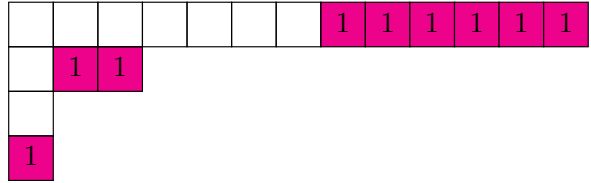
\begin{displaymath}
	\ydiagram{7,1,1}*[*(magenta) 1]{7+6,1+2,0,0+1}
\end{displaymath}
In particular, we can notice that the number of $1$'s occurs in the first row  is at least $3$ and at most $m-3$. Hence, the number of $1$'s in the second row is at least $2$. Therefore, we may replace $\beta$ with $(m-2,2)$ as shown below.
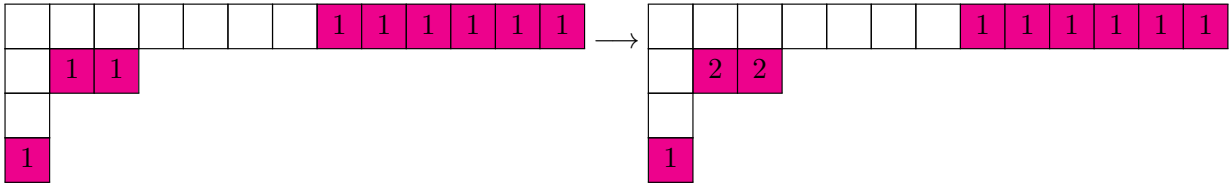
\begin{displaymath}
	\ydiagram{7,1,1}*[*(magenta) 1]{7+6,1+2,0,0+1}
	\longrightarrow
	\ydiagram{7,1,1}*[*(magenta) 1]{7+6,0,0,0+1}*[*(magenta) 2]{0,1+2,0}
\end{displaymath}

Since the $\lambda/\alpha$ has at least two columns, $\beta\neq (1^m)$.
Suppose that $\beta=(m-1,1)$.
We may replace $\beta$ with $(m-2,2)$ as shown below.

There are three possible diagrams.
First, there are two cells in the second column of $\lambda/\alpha$.

If there is a cell in the first column of $\lambda/\alpha$, then we may replace the $1$ in the first column with $2$.
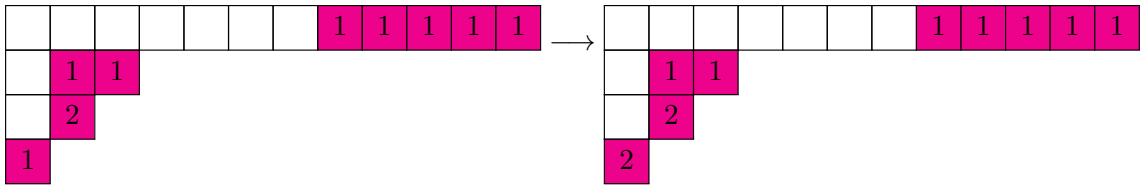
\begin{displaymath}
	\ydiagram{7,1,1}*[*(magenta) 1]{7+5,1+2,0,0+1}*[*(magenta) 2]{0,0,1+1}*[*(magenta) 3]{0,0,0,0}
	\longrightarrow
	\ydiagram{7,1,1}*[*(magenta) 1]{7+5,1+2,0,0}*[*(magenta) 2]{0,0,1+1,0+1}*[*(magenta) 3]{0,0,0,0}
\end{displaymath}
Otherwise, we may replace the rightmost $1$ in the second row with $2$.
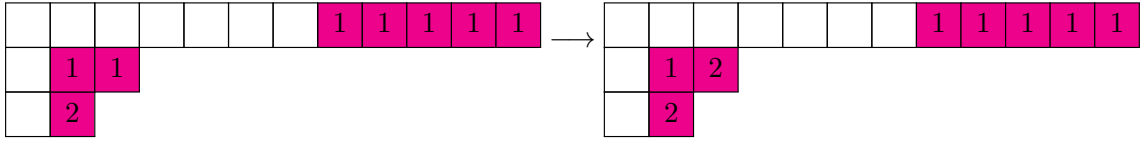
\begin{displaymath}
	\ydiagram{7,1,1}*[*(magenta) 1]{7+5,1+2,0,0}*[*(magenta) 2]{0,0,1+1}*[*(magenta) 3]{0,0,0,0}
	\longrightarrow
	\ydiagram{7,1,1}*[*(magenta) 1]{7+5,1+1,0,0}*[*(magenta) 2]{0,2+1,1+1}*[*(magenta) 3]{0,0,0,0}
\end{displaymath}
Second, there are two cells in the first column of $\lambda/\alpha$.
Since $\lambda_1\leq 2m-5$, the second row of $\lambda/\alpha$ has at least $1$ cell.
We may change the rightmost $1$ in the second row with $2$.
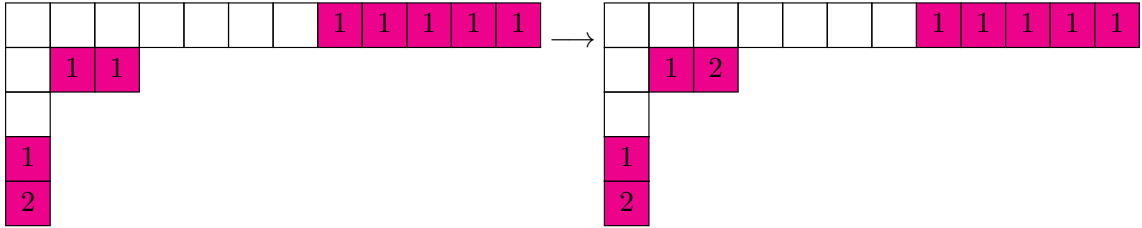
\begin{displaymath}
	\ydiagram{7,1,1}*[*(magenta) 1]{7+5,1+2,0,0+1}*[*(magenta) 2]{0,0,0,0,0+1}*[*(magenta) 3]{0,0,0,0}
	\longrightarrow
	\ydiagram{7,1,1}*[*(magenta) 1]{7+5,1+1,0,0+1}*[*(magenta) 2]{0,2+1,0,0,0+1}*[*(magenta) 3]{0,0,0,0}
\end{displaymath}
Third, there are two cells in the $(m-1)$st column of $\lambda/\alpha$.
Then there is only one possible diagram due to $\lambda_1\geq m$ and we may replace the rightmost $1$ in the second row with $2$.
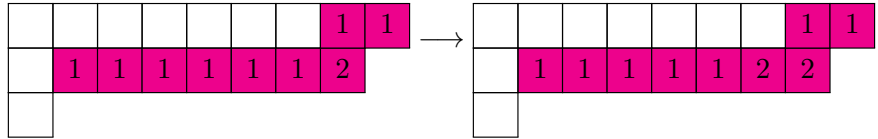
\begin{displaymath}
	\ydiagram{7,1,1}*[*(magenta) 1]{7+2,1+6,0,0}*[*(magenta) 2]{0,7+1}*[*(magenta) 3]{0,0,0,0}
	\longrightarrow
	\ydiagram{7,1,1}*[*(magenta) 1]{7+2,1+5,0,0}*[*(magenta) 2]{0,6+2}*[*(magenta) 3]{0,0,0,0}
\end{displaymath}
Finally, let us suppose that $\beta=\alpha=(m-2,1,1)$. We may replace $\beta$ with $(m-4,1,1,1)$ as shown below.
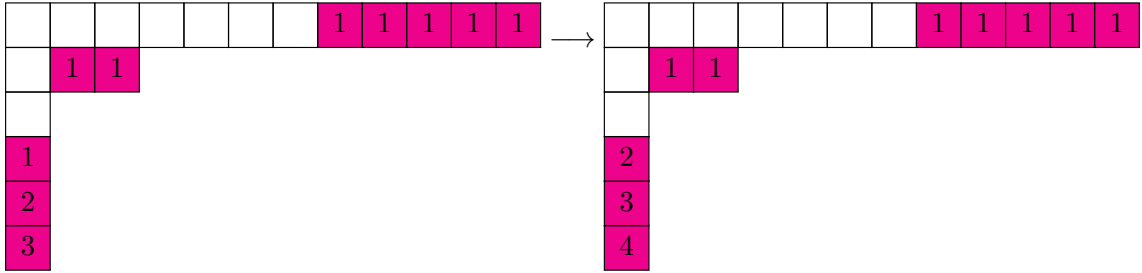
\begin{displaymath}
	\ydiagram{7,1,1}*[*(magenta) 1]{7+5,1+2,0,0+1}*[*(magenta) 2]{0,0,0,0,0+1}*[*(magenta) 3]{0,0,0,0,0,0+1}
	\longrightarrow
	\ydiagram{7,1,1}*[*(magenta) 1]{7+5,1+2,0}*[*(magenta) 2]{0,0,0,0+1}*[*(magenta) 3]{0,0,0,0,0+1}*[*(magenta) 4]{0,0,0,0,0,0+1}
\end{displaymath}
In the following diagram, we replaced the rightmost $1$ in the second row with $2$.
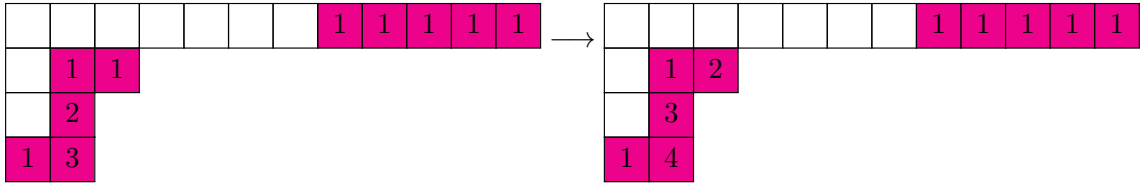
\begin{displaymath}
	\ydiagram{7,1,1}*[*(magenta) 1]{7+5,1+2,0,0+1}*[*(magenta) 2]{0,0,1+1}*[*(magenta) 3]{0,0,0,1+1}
	\longrightarrow
	\ydiagram{7,1,1}*[*(magenta) 1]{7+5,1+1,0,0+1}*[*(magenta) 2]{0,2+1,0}*[*(magenta) 3]{0,0,1+1}*[*(magenta) 4]{0,0,0,1+1}
\end{displaymath}
This completes the proof when $\lambda_1\geq m$.

\subsubsection*{Case III. $\lambda_1\geq \lambda_1'$ and $\lambda_1< m$.}

Suppose that $\lambda_1+\lambda_1'-1 \leq m$.
Then we may choose $\alpha\vdash m$ to be any partition of $m$ contained in $\lambda$ and contains the first row and the first column of $\lambda$.
Let $\beta$ be the dominant partition of $m$ contained in $\lambda/\alpha$.
Since $\lambda/\alpha$ has at most $\alpha_1-1$ many columns, $\alpha\neq\beta$.
Note also that $\alpha,\beta \notin A_m$.

\vspace{.3cm}
\noindent
Suppose that $\lambda_1+\lambda_1'-1> m$.
Let's split this case into two subcases.
\begin{itemize}
\item Let $\lambda_1<m-1$.
Let $\alpha,\beta$ be dominant partitions of $m$ contained in $\lambda$ and $\lambda/\alpha$ respectively.
If $\alpha\neq\beta$, then we are done.

Otherwise, suppose that $\alpha=\beta$.
In this case, we must have $\ell(\alpha)=2$ since $\lambda_1+\lambda_1'-1> m$.
\begin{itemize}
\item Let us start with the case that $\lambda_{4}>0$.
In this case, we may replace $\alpha$ with $(\alpha_1-1,\alpha_2,1)$ and $\beta$ with $(\alpha_1,\alpha_2-1,1)$ as shown below.
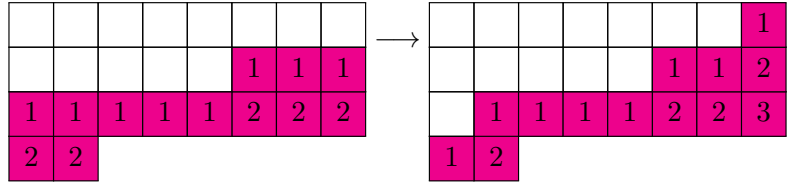
\begin{displaymath}
\ydiagram{8,5}*[*(magenta) 1]{0,5+3,0+5}*[*(magenta) 2]{0,0,5+3,0+2,0}*[*(magenta) 3]{0,0,0}
\longrightarrow
\ydiagram{7,5,1}*[*(magenta) 1]{7+1,5+2,1+4,0+1}*[*(magenta) 2]{0,7+1,5+2,1+1,0}*[*(magenta) 3]{0,0,7+1}
\end{displaymath}
\item Let $\lambda_4=0$. In this case, we may replace $\alpha$ with $(\alpha_1,\alpha_2-1,1)$ and $\beta$ with $(\alpha_1-1,\alpha_2+1,1)$ as shown below.
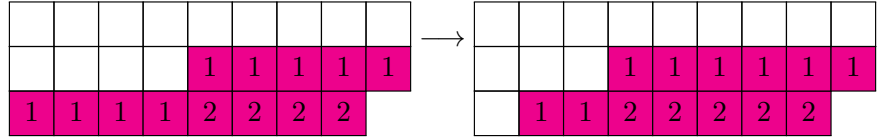
\begin{displaymath}
\ydiagram{9,4}*[*(magenta) 1]{0,4+5,0+4}*[*(magenta) 2]{0,0,4+4,0}*[*(magenta) 3]{0,0,0}
\longrightarrow
\ydiagram{9,3,1}*[*(magenta) 1]{0,3+6,1+2}*[*(magenta) 2]{0,0,3+5,0}*[*(magenta) 3]{0,0,0}
\end{displaymath}
\end{itemize}
\item Let $\lambda_1=m-1$.
Let us choose $\alpha=(m-2,1,1)$ and $\beta$ be the dominant partition of $m$ contained in $\lambda/\alpha$.
Note that $\beta\neq (m),(1^m),(2,1^{m-2})$.
If $\beta$ is different from $(m-1,1)$ and $\alpha$, we are done.

Otherwise, suppose that $\beta=(m-1,1)$.
If the first or the second column of $\lambda/\alpha$ contains two cells, then we may replace $\beta$ with $(m-2,2)$ by replacing the rightmost $1$ in the second row with $2$ as shown below.
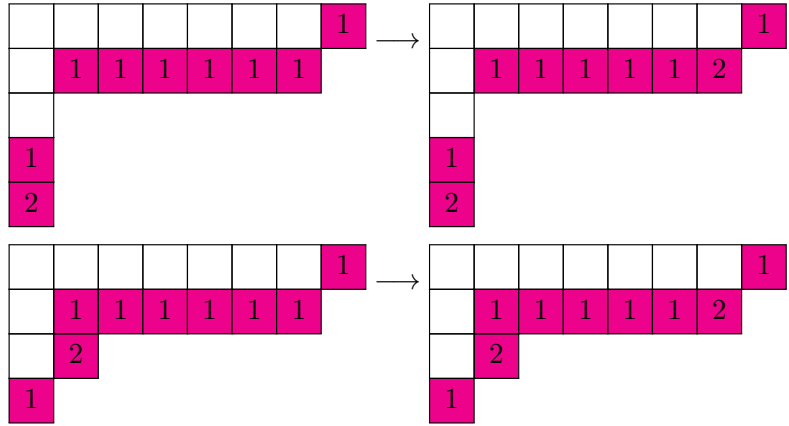
\begin{displaymath}
	\ydiagram{7,1,1}*[*(magenta) 1]{7+1,1+6,0,0+1,0}*[*(magenta) 2]{0,0,0,0,0+1}*[*(magenta) 3]{0,0,0,0}
	\longrightarrow
	\ydiagram{7,1,1}*[*(magenta) 1]{7+1,1+5,0,0+1,0}*[*(magenta) 2]{0,6+1,0,0,0+1}*[*(magenta) 3]{0,0,0,0}
\end{displaymath}
\begin{displaymath}
	\ydiagram{7,1,1}*[*(magenta) 1]{7+1,1+6,0,0+1,0}*[*(magenta) 2]{0,0,1+1,0}*[*(magenta) 3]{0,0,0,0}
	\longrightarrow
	\ydiagram{7,1,1}*[*(magenta) 1]{7+1,1+5,0,0+1,0}*[*(magenta) 2]{0,6+1,1+1,0}*[*(magenta) 3]{0,0,0,0}
\end{displaymath}
If the $m-1$st column of $\lambda/\alpha$ contains two cells, then we may replace $\beta$ with $(m-2,2)$ as shown below.
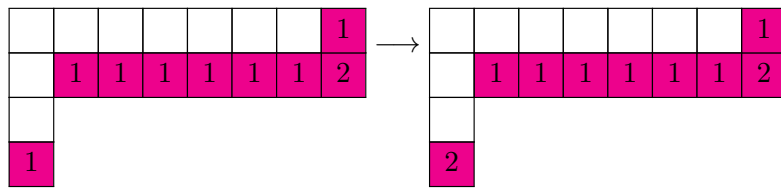
\begin{displaymath}
	\ydiagram{7,1,1}*[*(magenta) 1]{7+1,1+6,0,0+1,0}*[*(magenta) 2]{0,7+1}*[*(magenta) 3]{0,0,0,0}
	\longrightarrow
	\ydiagram{7,1,1}*[*(magenta) 1]{7+1,1+6,0,0,0}*[*(magenta) 2]{0,7+1,0,0+1}*[*(magenta) 3]{0,0,0,0}
\end{displaymath}

Finally, let us assume that $\beta=\alpha=(m-2,1,1)$.
In this case, either the first or the second column of $\lambda/\alpha$ contains three cells.
Then incrementing the entries in the first column of $\lambda/\alpha$ by $1$ yields a $\beta$ equal to $(m-3,1,1,1)$ in the former case and $(m-3,2,1)$ in the latter case.
\end{itemize}
This completes the proof of the latter statement of the lemma.
\end{proof}
A simple corollary of the above lemma is the following.
\begin{lemma}
	\label{lemma:p_p_except_standard_and_its_prime}
	Let $m\geq 11$ be an integer.
	\begin{displaymath}
		\sum\limits_{\substack{\alpha,\beta\vdash m\\
		\alpha\neq (m-1,1),(2,1^{m-2})\\
		\beta\neq (m-1,1),(2,1^{m-2})}} s_\alpha s_\beta \geq s_\lambda
	\end{displaymath}
	for all partitions $\lambda$ of $2m$.
\end{lemma}
\begin{proof}
	The theorem easily follows from the latter statement of Lemma~\ref{lemma:base} except for the partitions $\lambda$ with $\lambda_1>2m-5$, $\lambda_1'>2m-5$ or $\lambda$ is one of $(2m),(2m-1,1), (2m-2,2), (2m-2,1,1), (m,m), (3,1^{2m-3}),(2^2,1^{2m-4}),\\ (2,1^{2m-2}),(1^{2m})$.
	In these cases, the theorem can be verified by direct computation.
\end{proof}
The following lemma is one of the keys to prove the main theorem.
\begin{lemma}
	\label{lemma:3m_distinct}
	Let $m\geq 11$ be a positive integer.
	Then
	\begin{equation}
		\sum\limits_{\substack{\alpha,\beta,\gamma\neq (m),(1^m) \\
		\alpha\neq\beta,\beta\neq\gamma,\alpha\neq\gamma}} s_\alpha s_\beta s_\gamma \geq s_\lambda
	\end{equation}
	for all partitions $\lambda$ of $3m$ with $\lambda_1,\lambda_1'\leq 3m-6$.
\end{lemma}
\begin{proof}
	Let $\lambda$ be a partition of $3m$ with $\lambda_1\geq \lambda_1'$ and $\lambda_1\leq 3m-6$.
	In the whole proof, $\alpha,\beta,\gamma$ will denote partitions of $m$.
	Throughout the proof, we are free to replace the triple $(\alpha,\beta,\gamma)$
by another triple of pairwise distinct partitions of $m$ whenever the
Littlewood--Richardson condition remains satisfied.

	\subsubsection*{Case I. $\lambda$ is a hook or a two row partition.}
	First, let us start with the case when $\lambda$ is a hook.
	Then $\lambda=(m+k,1^{2m-k})$, where $k\geq m/2\geq 6$.
	If $\lambda_1\geq 2m-1$, then we may choose $\alpha=(m-1,1)$, $\beta=(m-2,1,1)$ and $\gamma=(k-m+3,1^{2m-k-3})$.
	Otherwise, we may choose $\alpha=(m-1,1)$, $\beta=(k-2,1^{m-k+2})$ and $\gamma=(2,1^{m-2})$.
	We shall consider the case when $\lambda$ is a two row partition.
	Then $\lambda=(m+k,m-k)$ where $k\geq m/2\geq 6$.
	One verifies directly using the Littlewood--Richardson rule that
$s_\lambda \ge s_\alpha s_\beta s_\gamma$, where $\alpha=(m-1,1)$, $\beta=(m-2,2)$ and $\gamma=(m-3,3)$.


	\vspace{0.2cm}
	From now on, we may assume that $\lambda$ is neither a hook nor a two row partition.
	\subsubsection*{Case II. Suppose that $\lambda_1\leq m-1.$}	
If $\lambda_1+\lambda_1'-1\leq m$, then let us choose $\alpha$ to be any partition contained in $\lambda$ and contains all the cells in the first row and the first column of $\lambda$.
	Since $m\ge \lambda_1\geq \lambda_1'$, $s_{\lambda/\alpha}\geq s_\mu$ for some partition $\mu$ of $2m$ such that $\mu_1\leq m-1, \mu_1'\leq m-1$.
	By Lemma~\ref{lemma:base}, there exist two partitions $\beta,\gamma$ such that $s_{\mu}\geq s_\beta s_\gamma$.
	Therefore, we have $s_\lambda\geq s_\alpha s_\beta s_\gamma$.
	\vspace{0.2cm}

\noindent Otherwise, $\lambda_1+\lambda_1'-1> m$.
	We may choose $\alpha$ to be the dominant partition (see Definition~\ref{definition:dominant_partition}) contained in $\lambda$.
	Note that $\alpha_2<\lambda_1$.
	Let $\beta$ be the dominant partition contained in $\lambda/\alpha$.
	
	Suppose that $\beta\neq \alpha$.
	Since $s_{\lambda/\alpha}\geq s_\mu$ for some partition $\mu$ of $2m$ such that $\mu_1\leq m-1, \mu_1'\leq m-1$, and the dominant partition $\beta$ contained in $\mu$ is less than $\alpha$ in dominance order. 
	By Lemma~\ref{lemma:base}, there exist two distinct partitions $\beta,\gamma$ such that $s_{\mu}\geq s_\beta s_\gamma$.
	Also, $\alpha\neq \beta$ and $\alpha\neq \gamma$.
	Therefore, we have $s_\lambda\geq s_\alpha s_\beta s_\gamma$.
	
	{ Suppose that $\alpha=\beta$.}
	In this case, we must have $\lambda_1=\lambda_2$ and $\alpha_2<\lambda_2$.
	Therefore, $\alpha_3=0$.
	The diagram of $\alpha$ and $\beta$ inside $\lambda$ looks like the following.
	\begin{displaymath}
		\ydiagram{8,5}*[*(magenta) 1]{0,5+3,0+5}*[*(magenta) 2]{0,0,5+3,0+2,0}*[*(magenta) 3]{0,0,0}
	\end{displaymath}
	
	\subsubsection*{Case II.1 Suppose that $\lambda_3=\alpha_2$.}
	Then we may replace $\beta$ with $(\alpha_1,\alpha_2-1,1)$ as shown below.
	\begin{displaymath}
		\ydiagram{8,5}*[*(magenta) 1]{0,5+3,0+5}*[*(magenta) 2]{0,0,0,0+5,0}*[*(magenta) 3]{0,0,0}
		\longrightarrow
		\ydiagram{8,5}*[*(magenta) 1]{0,5+3,0+5}*[*(magenta) 2]{0,0,0,0+4,0}*[*(magenta) 3]{0,0,0,0,0+1}*[*(brown) ]{0,0,0,4+1}
	\end{displaymath}
	Note that $\alpha_2=1$ is not possible in this case.
	Since the number of columns of $\lambda/\alpha/\beta$ is less than $\alpha_1$, the dominant partition $\gamma$ contained in $\lambda/\alpha/\beta$ is not equal to $\alpha$ and $\beta$, and we are done. 

	From now on, we may assume that $\lambda_3>\alpha_2$.
	\subsubsection*{Case II.2 Suppose that the number of cells with entries $2$ in the fourth row of $\lambda/\alpha$ is $0$.}
	Since $\lambda_1+\lambda_1'-1> m$ and $\ell(\lambda)\geq 4$, $\alpha_2\geq 4$. At the same time, $\beta_2=\alpha_2$ implies that $\alpha_1-\alpha_2\geq 4.$

	\begin{displaymath}
		\ydiagram{10,5}*[*(magenta) 1]{0,5+5,0+5}*[*(magenta) 2]{0,0,5+3,0,0}*[*(magenta) 3]{0,0,0}
	\end{displaymath}
	There are at least two pairs of partitions $(\alpha,\beta)$ possible in this case, 
	namely the pairs $((\alpha_1,\alpha_2-1,1),(\alpha_1-1,\alpha_2+1)),((\alpha_1,\alpha_2-2,2),(\alpha_1-2,\alpha_2+2))$.
	\begin{displaymath}
		\ydiagram{10,5}*[*(magenta) 1]{0,5+5,1+3}*[*(magenta) 2]{0,0,5+3,0}*[*(green) 1]{0,4+1,0}*[*(green) 2]{0,4+1,4+1}*[*(yellow) ]{0,0,0+1}
		\hspace{2cm}
		\ydiagram{10,5}*[*(magenta) 1]{0,5+5,2+1}*[*(magenta) 2]{0,0,5+3,0}*[*(green) 1]{0,3+2,0}*[*(green) 2]{0,0,3+2}*[*(yellow) ]{0,0,0+2}
	\end{displaymath}
	Now let $\gamma$ be the dominant partition contained in $\lambda/\alpha/\beta$.
	Then appending $\gamma$ to one of the above pairs $(\alpha,\beta)$ where $\gamma\neq \alpha$ and $\gamma\neq \beta$ gives the desired result.

\subsubsection*{Case II.3 Suppose that there is a $2$ in the fourth row of $\lambda/\alpha$ for $\beta$.}
We may start with the case where $\lambda_3<\alpha_1$.
In this case, we may replace $\beta$ with $(\alpha_1,\alpha_2-1,1)$ by replacing the rightmost $2$ in the third row with $3$.
The reverse row reading word will be a lattice permutation since $\lambda_3<\alpha_1$.
Now let $\gamma$ be the dominant partition contained in $\lambda/\alpha/\beta$.
Then the triple $(\alpha,\beta,\gamma)$ gives the desired result.

Now we may assume that $\lambda_3 = \alpha_1$.

	\begin{displaymath}
		\ydiagram{10,6}*[*(magenta) 1]{0,6+4,0+6}*[*(magenta) 2]{0,0,6+4,0+2,0}*[*(magenta) 3]{0,0,0}
	\end{displaymath}
In this case, we have at least two possible pairs of partitions $(\alpha,\beta)$, namely the pairs $((\alpha_1,\alpha_2-1,1),(\alpha_1-1,\alpha_2+1)),((\alpha_1-1,\alpha_2,1),(\alpha_1-1,\alpha_2-1,2))$.
	\begin{displaymath}
		\ydiagram{10,5,1}*[*(magenta) 1]{0,5+5,1+4}*[*(magenta) 2]{0,0,5+5,0+2,0}*[*(magenta) 3]{0,0,0}*[*(yellow) ]{0,0,0+1}
		\hspace{2cm}
		\ydiagram{9,6,1}*[*(magenta) 1]{9+1,6+3,1+5}*[*(magenta) 2]{0,9+1,6+3,0+2,0}*[*(magenta) 3]{0,0,9+1,0}*[*(yellow) ]{0,0,0+1}
	\end{displaymath}

	Now the dominant partition $\gamma$ contained in $\lambda/\alpha/\beta$ appended to one of the above pairs $(\alpha,\beta)$ where $\gamma\neq \alpha$ and $\gamma\neq \beta$ gives the desired result.
This completes the proof when $\lambda_1\leq m-1$.
\subsubsection*{Case III. Suppose that $\lambda_1\geq m$.}

We may choose $\alpha=(p-1,1)$ and $\beta$ be the dominant partition of $m$ contained in $\lambda/\alpha$.

\textbf{1. Suppose that $\beta=(m)$.}
\begin{displaymath}
	\ydiagram{7,1}*[*(magenta) 1]{7+3,1+3,0+1}
\end{displaymath}

\textbf{1.1 If both second and the third row of $\lambda/\alpha$ contains $1$, then we may replace $\beta$ by $(m-2,1,1)$ as shown below.}
\begin{displaymath}
	\ydiagram{7,1}*[*(magenta) 1]{7+3,1+2,0}*[*(magenta) 2]{0,3+1,0}*[*(magenta) 3]{0,0,0+1}
\end{displaymath}
Let $\gamma$ be the dominant partition of $m$ contained in $\lambda/\alpha/\beta$.
Since $\lambda/\alpha/\beta$ has at most $m-1$ columns, therefore $\gamma$ is not equal to $(m)$.

If $\gamma=(1^m)$, then we may change the appearance of $\beta$ and replace $\gamma$ with $(2,1^{m-2})$ as shown below.
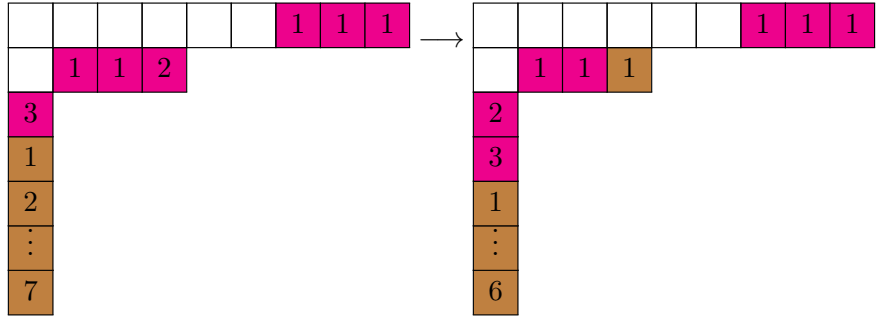
\begin{displaymath}
	\ydiagram{6,1}*[*(magenta) 1]{6+3,1+2,0}*[*(magenta) 2]{0,3+1,0}*[*(magenta) 3]{0,0,0+1}*[*(brown) 1]{0,0,0,0+1}*[*(brown) 2]{0,0,0,0,0+1}*[*(brown) \vdots]{0,0,0,0,0,0+1}*[*(brown) 7]{0,0,0,0,0,0,0+1}
	\longrightarrow
	\ydiagram{6,1}*[*(magenta) 1]{6+3,1+2,0}*[*(magenta) 2]{0,0,0+1}*[*(magenta) 3]{0,0,0,0+1}*[*(brown) 1]{0,3+1,0,0,0+1}*[*(brown) \vdots]{0,0,0,0,0,0+1}*[*(brown) 6]{0,0,0,0,0,0,0+1}
\end{displaymath}

Suppose that $\gamma=\alpha=(m-1,1)$.
If $\lambda_2'=4$, then we may change the appearance of $\beta$ and replace $\gamma$ with $(m-2,2)$ as shown below.
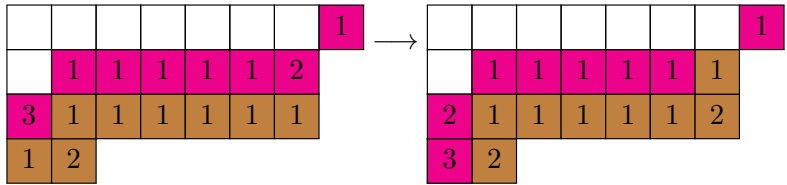
\begin{displaymath}
	\ydiagram{7,1}*[*(magenta) 1]{7+1,1+5,0}*[*(magenta) 2]{0,6+1,0}*[*(magenta) 3]{0,0,0+1}*[*(brown) 1]{0,0,1+6,0+1}*[*(brown) 2]{0,0,0,1+1}
	\longrightarrow
	\ydiagram{7,1}*[*(magenta) 1]{7+1,1+5,0}*[*(magenta) 2]{0,0,0+1}*[*(magenta) 3]{0,0,0,0+1}*[*(brown) 1]{0,6+1,1+5,0+1}*[*(brown) 2]{0,0,6+1,1+1}
\end{displaymath}
If $\lambda_1'=5$, then we may replace $\beta$ with $(m-2,2)$ and $\gamma$ with $(m-2,1,1)$ as shown below.
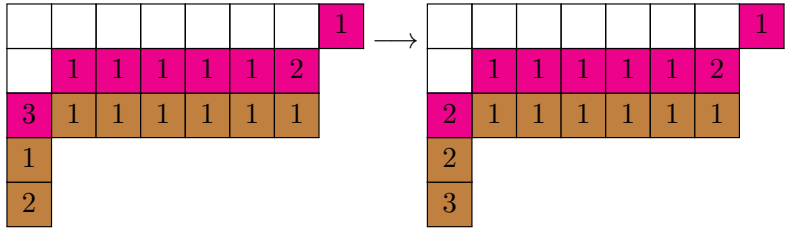
\begin{displaymath}
	\ydiagram{7,1}*[*(magenta) 1]{7+1,1+5,0}*[*(magenta) 2]{0,6+1,0}*[*(magenta) 3]{0,0,0+1}*[*(brown) 1]{0,0,1+6,0+1}*[*(brown) 2]{0,0,0,0,0+1}
	\longrightarrow
	\ydiagram{7,1}*[*(magenta) 1]{7+1,1+5,0}*[*(magenta) 2]{0,6+1,0+1}*[*(brown) 1]{0,0,1+6}*[*(brown) 2]{0,0,0,0+1}*[*(brown) 3]{0,0,0,0,0+1}
\end{displaymath}
Finally, suppose that $\gamma=\beta=(m-2,1,1)$.
In this case, we may replace $\gamma$ with $(m-3,1,1,1)$ as shown below.
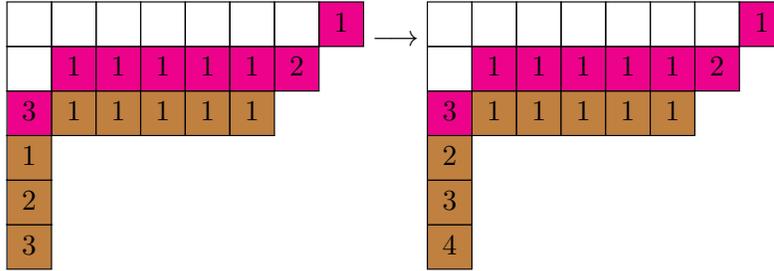
\begin{displaymath}
	\ydiagram{7,1}*[*(magenta) 1]{7+1,1+5,0}*[*(magenta) 2]{0,6+1,0}*[*(magenta) 3]{0,0,0+1}*[*(brown) 1]{0,0,1+5,0+1}*[*(brown) 2]{0,0,0,0,0+1}*[*(brown) 3]{0,0,0,0,0,0+1}
	\longrightarrow
	\ydiagram{7,1}*[*(magenta) 1]{7+1,1+5,0}*[*(magenta) 2]{0,6+1,0}*[*(magenta) 3]{0,0,0+1}*[*(brown) 1]{0,0,1+5,0}*[*(brown) 2]{0,0,0,0+1}*[*(brown) 3]{0,0,0,0,0+1}*[*(brown) 4]{0,0,0,0,0,0+1}
\end{displaymath}
Since $\lambda$ is not  partition, $\lambda/\alpha$ has $1$ in the third row but not in the second row is not possible.
Therefore, we may now consider the case when $\lambda/\alpha$ has $1$ in the second row but not in the third row.
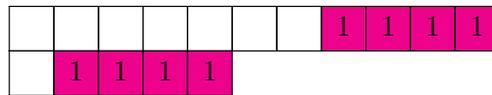
\begin{displaymath}
	\ydiagram{7,1}*[*(magenta) 1]{7+4,1+4,}
\end{displaymath}

Note that the first row of $\lambda/\alpha$ contains at least two $1$'s.

\textbf{1.2 Suppose that the second row of $\lambda/\alpha$ contains at least two $1$'s.}
We may replace $\beta$ with $(m-2,2)$ as shown below.
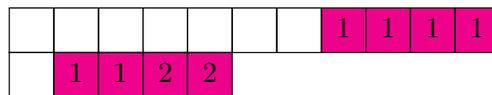
\begin{displaymath}
	\ydiagram{7,1}*[*(magenta) 1]{7+4,1+2,}*[*(magenta) 2]{0,3+2}
\end{displaymath}
Let $\gamma$ be the dominant partition of $m$ contained in $\lambda/\alpha/\beta$.
If $\gamma\notin \{(m),(1^m),\alpha,\beta\}$, then we are done.
Otherwise, suppose that $\gamma=(m)$.
Then we may replace $\alpha$ with $(m-2,1,1)$, $\beta$ with $(m-2,2)$ and $\gamma$ with $(m-1,1)$ as shown below. Note that $\ell(\lambda)>2$ has been used here.
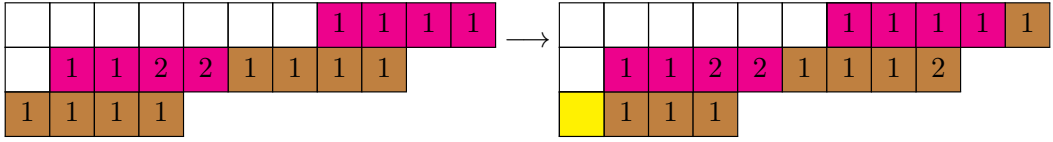
\begin{displaymath}
	\ydiagram{7,1}*[*(magenta) 1]{7+4,1+2}*[*(magenta) 2]{0,3+2}*[*(brown) 1]{0,5+4,0+4}
	\longrightarrow
	\ydiagram{7,1}*[*(yellow) ]{0,0,0+1}*[*(magenta) 1]{6+4,1+2}*[*(magenta) 2]{0,3+2,0+1}*[*(brown) 1]{10+1,5+3,1+3}*[*(brown) 2]{0,8+1,0}
\end{displaymath}
Suppose that $\gamma=(1^m)$. In this case, we may replace $\beta$ with $(m-2,1,1)$ and $\gamma$ with $(2,1^{m-2})$ as shown below.
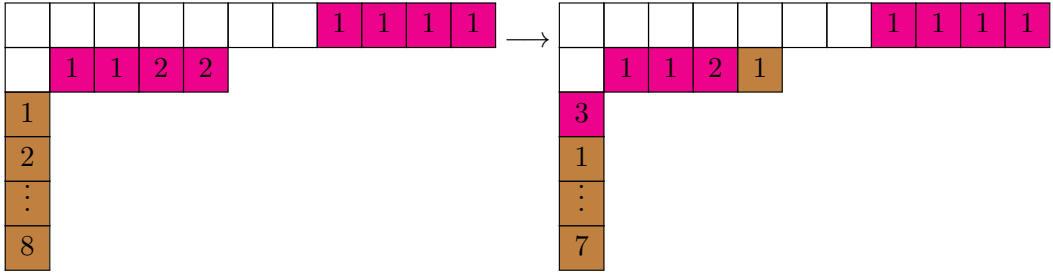
\begin{displaymath}
	\ydiagram{7,1}*[*(magenta) 1]{7+4,1+2}*[*(magenta) 2]{0,3+2}*[*(brown) 1]{0,0,0+1}*[*(brown) 2]{0,0,0,0+1}*[*(brown) \vdots]{0,0,0,0,0+1}*[*(brown) 8]{0,0,0,0,0,0+1}
	\longrightarrow
	\ydiagram{7,1}*[*(magenta) 1]{7+4,1+2}*[*(magenta) 2]{0,3+1}*[*(magenta) 3]{0,0,0+1}*[*(brown) 1]{0,4+1,0,0+1}*[*(brown) 2]{0,0,0,0}*[*(brown) \vdots]{0,0,0,0,0+1}*[*(brown) 7]{0,0,0,0,0,0+1}
\end{displaymath}
Suppose that $\gamma=\alpha=(m-1,1)$.
Then we may replace $\alpha$ with $(m-2,1,1)$, $\beta$ with $(m-2,2)$ and $\gamma$ with $(m-1,1)$ as shown below.
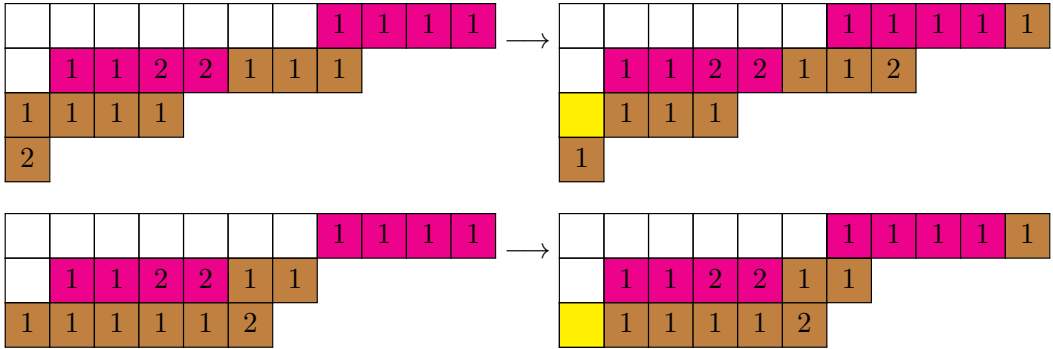
\begin{displaymath}
	\ydiagram{7,1}*[*(magenta) 1]{7+4,1+2}*[*(magenta) 2]{0,3+2}*[*(brown) 1]{0,5+3,0+4}*[*(brown) 2]{0,0,0,0+1}
	\longrightarrow
	\ydiagram{7,1}*[*(yellow) ]{0,0,0+1}*[*(magenta) 1]{6+4,1+2}*[*(magenta) 2]{0,3+2,0+1}*[*(brown) 1]{10+1,5+2,1+3,0+1}*[*(brown) 2]{0,7+1,0}
\end{displaymath}
\begin{displaymath}
	\ydiagram{7,1}*[*(magenta) 1]{7+4,1+2}*[*(magenta) 2]{0,3+2}*[*(brown) 1]{0,5+2,0+5}*[*(brown) 2]{0,0,5+1}
	\longrightarrow
	\ydiagram{7,1}*[*(yellow) ]{0,0,0+1}*[*(magenta) 1]{6+4,1+2}*[*(magenta) 2]{0,3+2,0+1}*[*(brown) 1]{10+1,5+2,1+4}*[*(brown) 2]{0,0,5+1,0}
\end{displaymath}
Suppose that $\gamma=\beta=(m-2,2)$.
If $\lambda_1'=\lambda_2'=4$, then we may replace $\gamma$ with $(m-3,2,1)$ as shown below.
When there is a cell with entry $1$ in the second row of $\lambda/\alpha/\beta$,
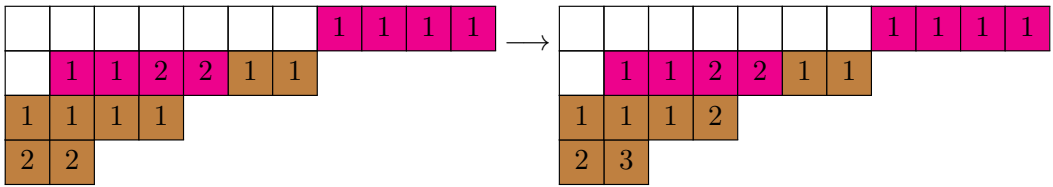
\begin{displaymath}
	\ydiagram{7,1}*[*(magenta) 1]{7+4,1+2}*[*(magenta) 2]{0,3+2}*[*(brown) 1]{0,5+2,0+4}*[*(brown) 2]{0,0,0,0+2}
	\longrightarrow
	\ydiagram{7,1}*[*(magenta) 1]{7+4,1+2}*[*(magenta) 2]{0,3+2}*[*(brown) 1]{0,5+2,0+3}*[*(brown) 2]{0,0,3+1,0+1}*[*(brown) 3]{0,0,0,1+1}
\end{displaymath}
When there is no cell with entry $1$ in the second row of $\lambda/\alpha/\beta$,
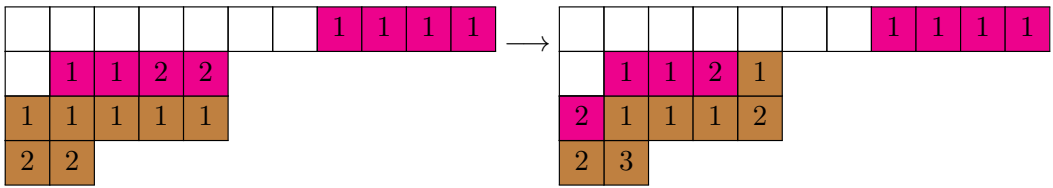
\begin{displaymath}
	\ydiagram{7,1}*[*(magenta) 1]{7+4,1+2}*[*(magenta) 2]{0,3+2}*[*(brown) 1]{0,0,0+5}*[*(brown) 2]{0,0,0,0+2}
	\longrightarrow
	\ydiagram{7,1}*[*(magenta) 1]{7+4,1+2}*[*(magenta) 2]{0,3+1,0+1}*[*(brown) 1]{0,4+1,0+4}*[*(brown) 2]{0,0,4+1,0+1}*[*(brown) 3]{0,0,0,1+1}
\end{displaymath}
If the $k$th column ($k=2m-\lambda_1+1$) and the $(k+1)$st column of $\lambda/\alpha/\beta$ contains $2$, then we may replace $\beta$ with $(m-2,1,1)$ and $\gamma$ with $(m-3,3)$ as shown below.
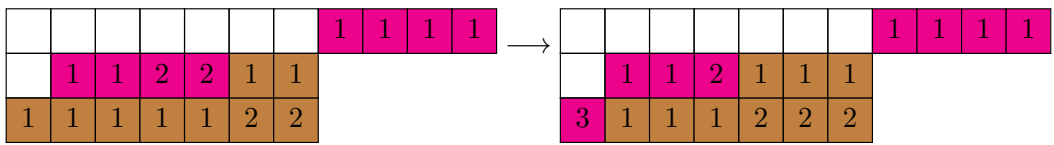
\begin{displaymath}
	\ydiagram{7,1}*[*(magenta) 1]{7+4,1+2}*[*(magenta) 2]{0,3+2}*[*(brown) 1]{0,5+2,0+5}*[*(brown) 2]{0,0,5+2}
	\longrightarrow
	\ydiagram{7,1}*[*(magenta) 1]{7+4,1+2}*[*(magenta) 2]{0,3+1}*[*(magenta) 3]{0,0,0+1}*[*(brown) 1]{0,4+3,1+3}*[*(brown) 2]{0,0,4+3}
\end{displaymath}
If the first column and the $k$th column of $\lambda/\alpha/\beta$ contains $2$, then we may replace $\beta$ with $(m-2,1,1)$ and $\gamma$ with $(m-3,3)$ as shown below.
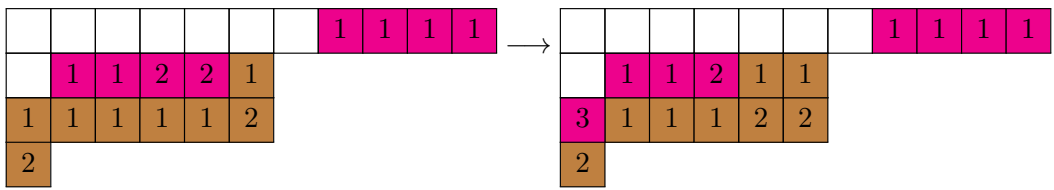
\begin{displaymath}
	\ydiagram{7,1}*[*(magenta) 1]{7+4,1+2}*[*(magenta) 2]{0,3+2}*[*(brown) 1]{0,5+1,0+5}*[*(brown) 2]{0,0,5+1,0+1}
	\longrightarrow
	\ydiagram{7,1}*[*(magenta) 1]{7+4,1+2}*[*(magenta) 2]{0,3+1}*[*(magenta) 3]{0,0,0+1}*[*(brown) 1]{0,4+2,1+3}*[*(brown) 2]{0,0,4+2,0+1}
\end{displaymath}
\textbf{1.3 Suppose that the second row of $\lambda/\alpha$ contains exactly one $1$.}
Then we may replace $\beta$ with $(m-2,1,1)$ as shown below.
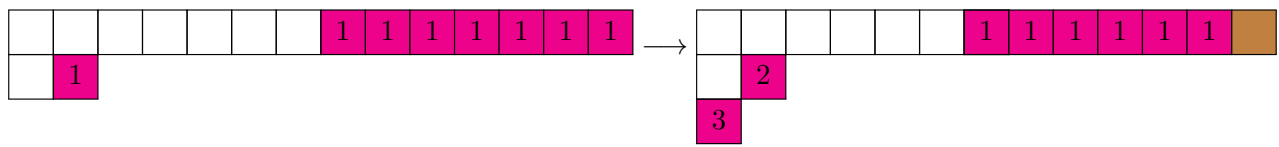
\begin{displaymath}
	\ydiagram{7,1}*[*(magenta) 1]{7+7,1+1}*[*(magenta) 2]{0}*[*(magenta) 3]{0,0}
	\longrightarrow
	\ydiagram{7,1}*[*(magenta) 1]{6+6,0,0}*[*(magenta) 2]{0,1+1}*[*(magenta) 3]{0,0,0+1}*[*(brown) ]{12+1,0,0}
\end{displaymath}
Let $\gamma$ be the dominant partition contained in $\lambda/\alpha/\beta$.
If $\gamma\notin\{(m),(1^m),\alpha,\beta\}$, then we are done.
Otherwise, suppose that $\gamma=(m)$.
Therefore, we may replace $\alpha$ with $(m-2,2)$, $\beta$ with $(m-3,3)$ and $\gamma$ with $(m-1,1)$ as shown below.
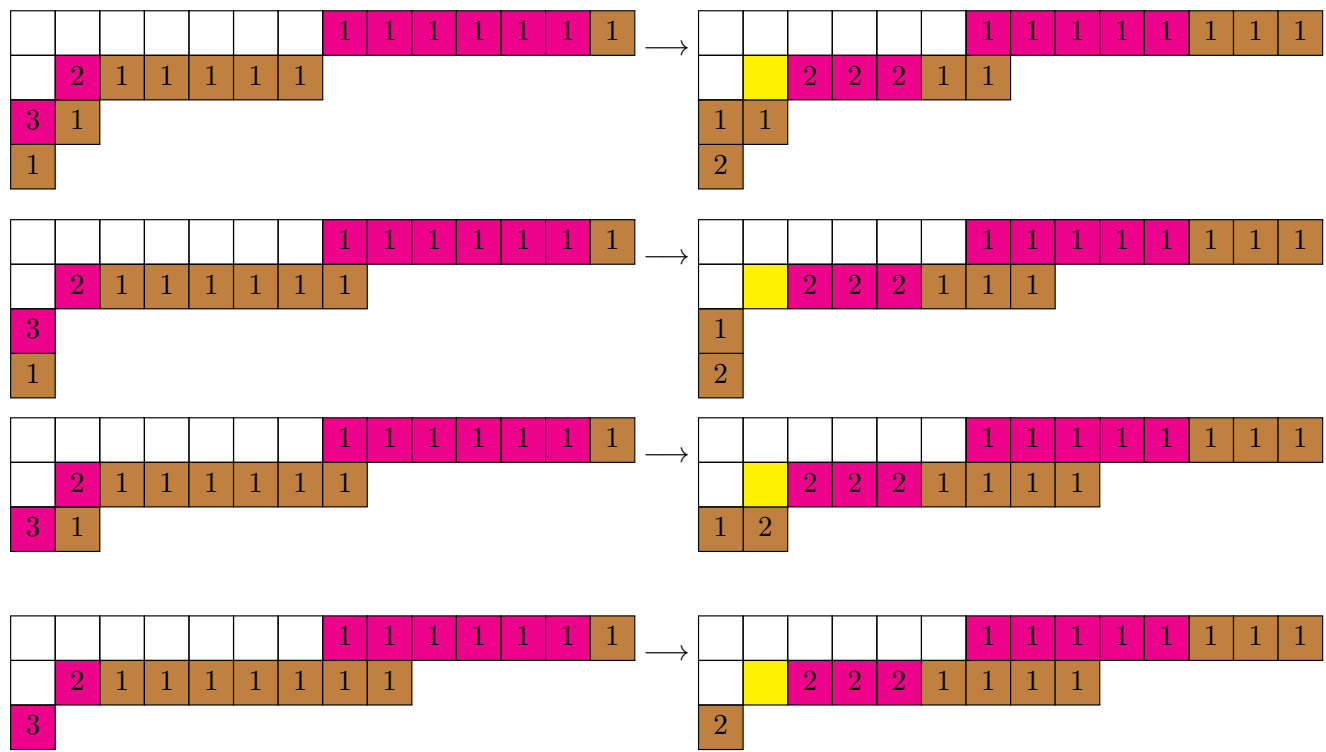
\begin{displaymath}
	\ydiagram{7,1}*[*(magenta) 1]{7+6,0,0}*[*(magenta) 2]{0,1+1}*[*(magenta) 3]{0,0,0+1}*[*(brown) 1]{13+1,2+5,1+1,0+1}
	\longrightarrow
	\ydiagram{6,1}*[*(magenta) 1]{6+5,0,0}*[*(yellow) ]{0,1+1,0}*[*(magenta) 2]{0,2+3}*[*(magenta) 3]{0,0,0}*[*(brown) 1]{11+3,5+2,0+2,0}*[*(brown) 2]{0,0,0,0+1}
\end{displaymath}
\begin{displaymath}
	\ydiagram{7,1}*[*(magenta) 1]{7+6,0,0}*[*(magenta) 2]{0,1+1}*[*(magenta) 3]{0,0,0+1}*[*(brown) 1]{13+1,2+6,0,0+1}
	\longrightarrow
	\ydiagram{6,1}*[*(magenta) 1]{6+5,0,0}*[*(yellow) ]{0,1+1,0}*[*(magenta) 2]{0,2+3}*[*(magenta) 3]{0,0,0}*[*(brown) 1]{11+3,5+3,0+1,0}*[*(brown) 2]{0,0,0,0+1}
\end{displaymath}
\begin{displaymath}
	\ydiagram{7,1}*[*(magenta) 1]{7+6,0,0}*[*(magenta) 2]{0,1+1}*[*(magenta) 3]{0,0,0+1}*[*(brown) 1]{13+1,2+6,1+1}
	\longrightarrow
	\ydiagram{6,1}*[*(magenta) 1]{6+5,0,0}*[*(yellow) ]{0,1+1,0}*[*(magenta) 2]{0,2+3}*[*(magenta) 3]{0,0,0}*[*(brown) 1]{11+3,5+4,0+1,0}*[*(brown) 2]{0,0,1+1}
\end{displaymath}
\begin{displaymath}
	\ydiagram{7,1}*[*(magenta) 1]{7+6,0,0}*[*(magenta) 2]{0,1+1}*[*(magenta) 3]{0,0,0+1}*[*(brown) 1]{13+1,2+7}
	\longrightarrow
	\ydiagram{6,1}*[*(magenta) 1]{6+5,0,0}*[*(yellow) ]{0,1+1,0}*[*(magenta) 2]{0,2+3}*[*(magenta) 3]{0,0,0}*[*(brown) 1]{11+3,5+4,0,0}*[*(brown) 2]{0,0,0+1}
\end{displaymath}
Since $2m-2$th column of $\lambda/\alpha/\beta$ is non-empty, $\gamma\neq (1^m)$.
Now suppose that $\gamma=\alpha=(m-1,1)$.
If $\lambda_1'=5$, then we may replace $\beta$ with $(m-2,2)$ and $\gamma$ with $(m-2,1,1)$ as shown below.
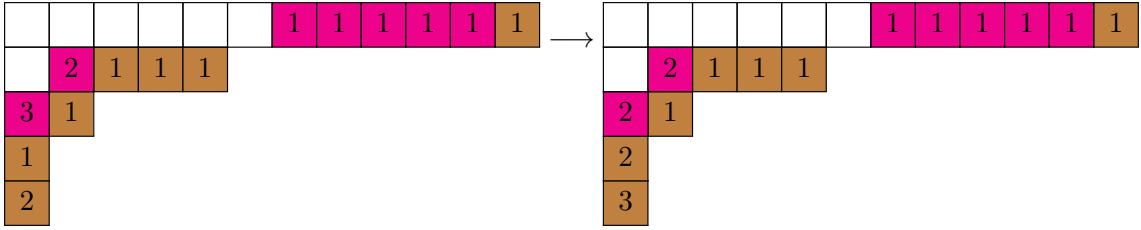
\begin{displaymath}
	\ydiagram{6,1}*[*(magenta) 1]{6+5,0,0}*[*(magenta) 2]{0,1+1}*[*(magenta) 3]{0,0,0+1}*[*(brown) 1]{11+1,2+3,1+1,0+1}*[*(brown) 2]{0,0,0,0,0+1}
	\longrightarrow
	\ydiagram{6,1}*[*(magenta) 1]{6+5,0,0}*[*(magenta) 2]{0,1+1,0+1}*[*(magenta) 3]{0,0,0}*[*(brown) 1]{11+1,2+3,1+1,0}*[*(brown) 2]{0,0,0,0+1}*[*(brown) 3]{0,0,0,0,0+1}
\end{displaymath}
If $\lambda_2'=4$, then we may replace $\gamma$ with $(m-2,2)$ as shown below.
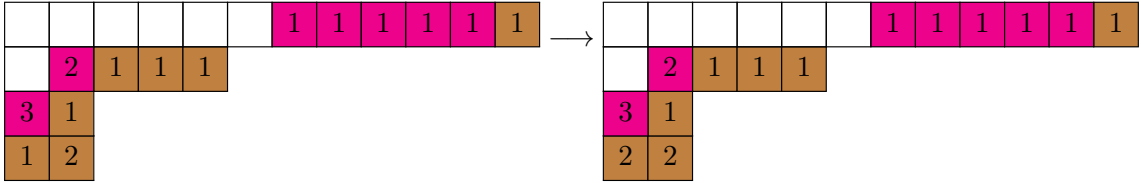
\begin{displaymath}
	\ydiagram{6,1}*[*(magenta) 1]{6+5,0,0}*[*(magenta) 2]{0,1+1}*[*(magenta) 3]{0,0,0+1}*[*(brown) 1]{11+1,2+3,1+1,0+1}*[*(brown) 2]{0,0,0,1+1}
	\longrightarrow
	\ydiagram{6,1}*[*(magenta) 1]{6+5,0,0}*[*(magenta) 2]{0,1+1,0}*[*(magenta) 3]{0,0,0+1}*[*(brown) 1]{11+1,2+3,1+1,0}*[*(brown) 2]{0,0,0,0+2}*[*(brown) 3]{0,0,0,0,0}
\end{displaymath}
If $\lambda_3'=3$, then we may replace $\gamma$ with $(m-2,2)$ as shown below.
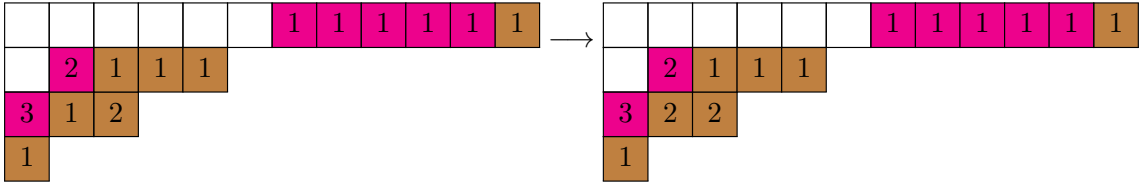
\begin{displaymath}
	\ydiagram{6,1}*[*(magenta) 1]{6+5,0,0}*[*(magenta) 2]{0,1+1}*[*(magenta) 3]{0,0,0+1}*[*(brown) 1]{11+1,2+3,1+1,0+1}*[*(brown) 2]{0,0,2+1}
	\longrightarrow
	\ydiagram{6,1}*[*(magenta) 1]{6+5,0,0}*[*(magenta) 2]{0,1+1,0}*[*(magenta) 3]{0,0,0+1}*[*(brown) 1]{11+1,2+3,0,0+1}*[*(brown) 2]{0,0,1+2}*[*(brown) 3]{0,0,0,0,0}
\end{displaymath}
Finally, suppose that $\gamma=\beta=(m-2,1,1)$.
Then $3$ of $\gamma$ must appear in the first column of $\lambda$.
We may replace $\gamma$ with $(m-3,1,1,1)$ as shown below.
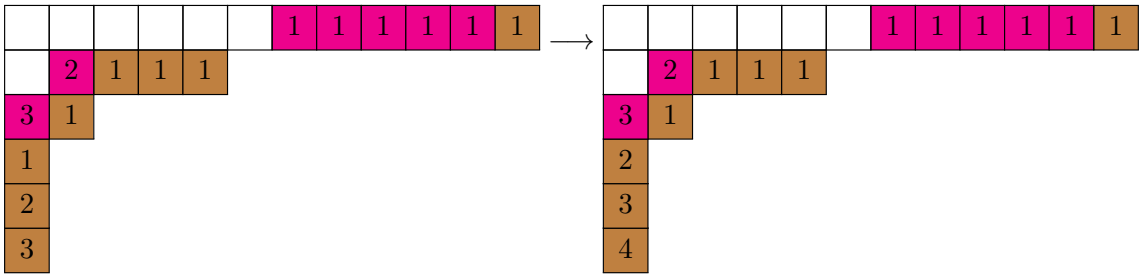
\begin{displaymath}
	\ydiagram{6,1}*[*(magenta) 1]{6+5,0,0}*[*(magenta) 2]{0,1+1}*[*(magenta) 3]{0,0,0+1}*[*(brown) 1]{11+1,2+3,1+1,0+1}*[*(brown) 2]{0,0,0,0,0+1}*[*(brown) 3]{0,0,0,0,0,0+1}
	\longrightarrow
	\ydiagram{6,1}*[*(magenta) 1]{6+5,0,0}*[*(magenta) 2]{0,1+1}*[*(magenta) 3]{0,0,0+1}*[*(brown) 1]{11+1,2+3,1+1}*[*(brown) 2]{0,0,0,0+1}*[*(brown) 3]{0,0,0,0,0+1}*[*(brown) 4]{0,0,0,0,0,0+1}
\end{displaymath}
Finally, let us consider the case where the there is no $1$ appears in the second row of $\lambda/\alpha$.
Equivalently, $\lambda_1\geq 2m-1$.
Since $\lambda$ is neither a hook nor a two row partition, we may replace $\beta$ with $(m-2,1,1)$ as shown below.
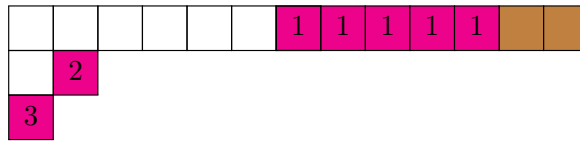
\begin{displaymath}
	\ydiagram{7,1}*[*(magenta) 1]{6+5,0,0}*[*(magenta) 2]{0,1+1}*[*(magenta) 3]{0,0,0+1}*[*(brown) ]{11+2,0,0}
\end{displaymath}
Let $\gamma$ be the dominant partition of $m$ contained in $\lambda/\alpha/\beta$.
If $\gamma\notin \{(m),(1^m),\alpha,\beta\}$, then we are done.
Otherwise, suppose that $\gamma=(m)$.
Then we may replace $\beta$ with $(m-2,2)$ and $\gamma$ with $(m-1,1)$ as shown below.

\noindent First case, when there is a cell in the first and the second column of $\lambda/\alpha/\beta$,
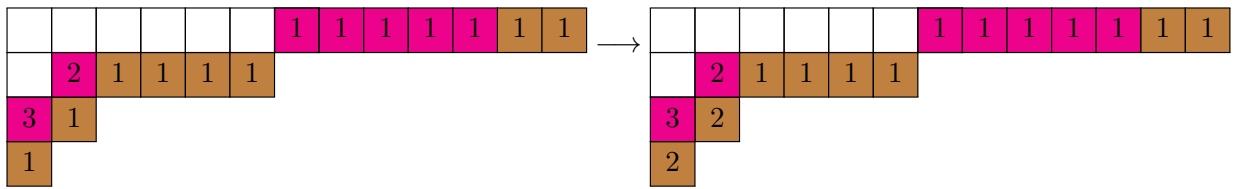
\begin{displaymath}
	\ydiagram{7,1}*[*(magenta) 1]{6+5,0,0}*[*(magenta) 2]{0,1+1}*[*(magenta) 3]{0,0,0+1}*[*(brown) 1]{11+2,2+4,1+1,0+1}
	\longrightarrow
	\ydiagram{7,1}*[*(magenta) 1]{6+5,0,0}*[*(magenta) 2]{0,1+1}*[*(magenta) 3]{0,0,0+1}*[*(brown) 1]{11+2,2+4}*[*(brown) 2]{0,0,1+1,0+1}
\end{displaymath}
Second case, when there is a cell in the first column and no cell in the second column of $\lambda/\alpha/\beta$,
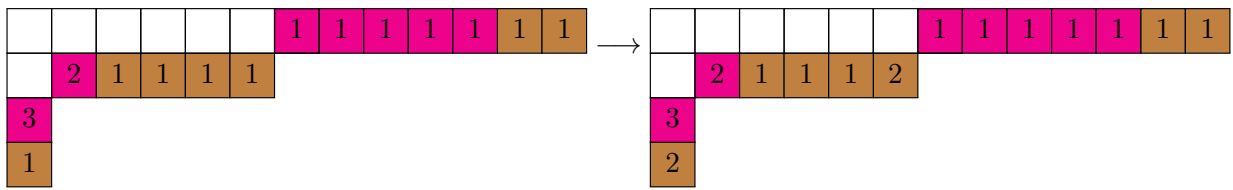
\begin{displaymath}
	\ydiagram{7,1}*[*(magenta) 1]{6+5,0,0}*[*(magenta) 2]{0,1+1}*[*(magenta) 3]{0,0,0+1}*[*(brown) 1]{11+2,2+4,0,0+1}
	\longrightarrow
	\ydiagram{7,1}*[*(magenta) 1]{6+5,0,0}*[*(magenta) 2]{0,1+1}*[*(magenta) 3]{0,0,0+1}*[*(brown) 1]{11+2,2+3,0,0}*[*(brown) 2]{0,5+1,0,0+1}
\end{displaymath}
Third case, when there is a cell in the second column and no cell in the first column of $\lambda/\alpha/\beta$,
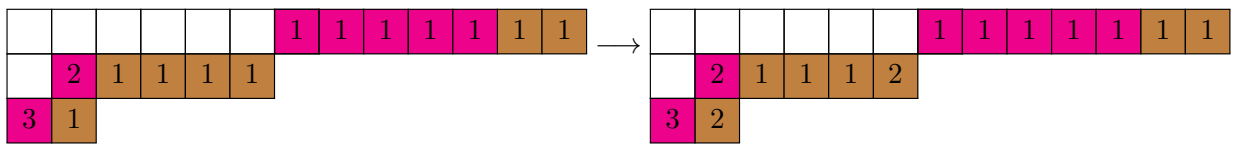
\begin{displaymath}
	\ydiagram{7,1}*[*(magenta) 1]{6+5,0,0}*[*(magenta) 2]{0,1+1}*[*(magenta) 3]{0,0,0+1}*[*(brown) 1]{11+2,2+4,1+1}
	\longrightarrow
	\ydiagram{7,1}*[*(magenta) 1]{6+5,0,0}*[*(magenta) 2]{0,1+1}*[*(magenta) 3]{0,0,0+1}*[*(brown) 1]{11+2,2+3,0,0}*[*(brown) 2]{0,5+1,1+1}
\end{displaymath}
Finally, when there is no cell in the first and the second column of $\lambda/\alpha/\beta$,
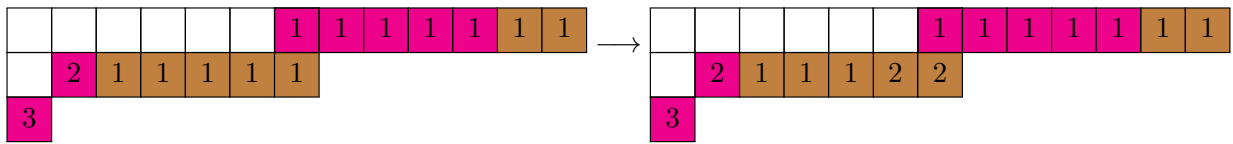
\begin{displaymath}
	\ydiagram{7,1}*[*(magenta) 1]{6+5,0,0}*[*(magenta) 2]{0,1+1}*[*(magenta) 3]{0,0,0+1}*[*(brown) 1]{11+2,2+5}
	\longrightarrow
	\ydiagram{7,1}*[*(magenta) 1]{6+5,0,0}*[*(magenta) 2]{0,1+1}*[*(magenta) 3]{0,0,0+1}*[*(brown) 1]{11+2,2+3,0,0}*[*(brown) 2]{0,5+2}
\end{displaymath}
Since $\lambda/\alpha/\beta$ has at least two non-empty columns, therefore $\gamma$ is not equal to $(1^m)$.
Now suppose that $\gamma=\alpha=(m-1,1)$.
If $\lambda_1'=5$, then we may replace $\beta$ with $(m-2,2)$ and $\gamma$ with $(m-2,1,1)$ as shown below.
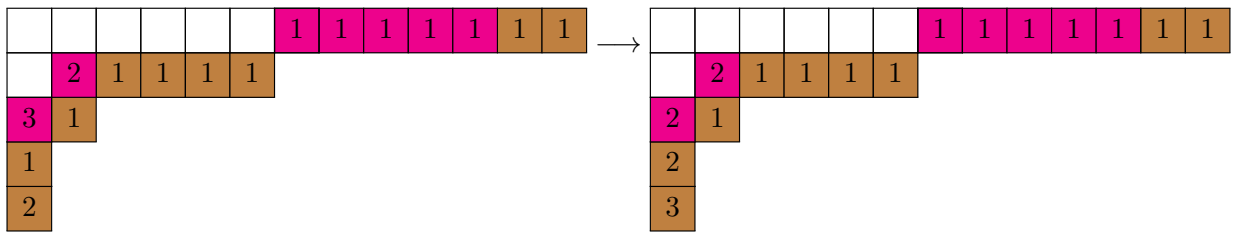
\begin{displaymath}
	\ydiagram{7,1}*[*(magenta) 1]{6+5,0,0}*[*(magenta) 2]{0,1+1}*[*(magenta) 3]{0,0,0+1}*[*(brown) 1]{11+2,2+4,1+1,0+1}*[*(brown) 2]{0,0,0,0,0+1}
	\longrightarrow
	\ydiagram{7,1}*[*(magenta) 1]{6+5,0,0}*[*(magenta) 2]{0,1+1}*[*(magenta) 2]{0,0,0+1}*[*(brown) 1]{11+2,2+4,1+1}*[*(brown) 2]{0,0,0,0+1}*[*(brown) 3]{0,0,0,0,0+1}
\end{displaymath}
If $\lambda_2'=4$, then we may replace $\gamma$ with $(m-2,2)$ as shown below.
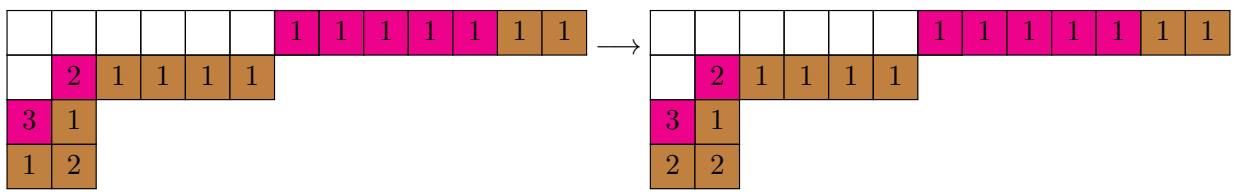
\begin{displaymath}
	\ydiagram{7,1}*[*(magenta) 1]{6+5,0,0}*[*(magenta) 2]{0,1+1}*[*(magenta) 3]{0,0,0+1}*[*(brown) 1]{11+2,2+4,1+1,0+1}*[*(brown) 2]{0,0,0,1+1}
	\longrightarrow
	\ydiagram{7,1}*[*(magenta) 1]{6+5,0,0}*[*(magenta) 2]{0,1+1}*[*(magenta) 3]{0,0,0+1}*[*(brown) 1]{11+2,2+4,1+1}*[*(brown) 2]{0,0,0,0+2}
\end{displaymath}
If $\lambda_3'=3$, then we may replace $\gamma$ with $(m-2,2)$ as shown below.
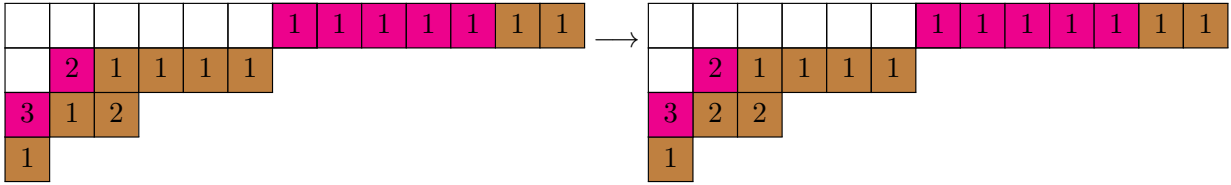
\begin{displaymath}
	\ydiagram{7,1}*[*(magenta) 1]{6+5,0,0}*[*(magenta) 2]{0,1+1}*[*(magenta) 3]{0,0,0+1}*[*(brown) 1]{11+2,2+4,1+1,0+1}*[*(brown) 2]{0,0,2+1}
	\longrightarrow
	\ydiagram{7,1}*[*(magenta) 1]{6+5,0,0}*[*(magenta) 2]{0,1+1}*[*(magenta) 3]{0,0,0+1}*[*(brown) 1]{11+2,2+4,0,0+1}*[*(brown) 2]{0,0,1+2}
\end{displaymath}
Now suppose that $\gamma=\beta=(m-2,1,1)$.
We may replace $\gamma$ with $(m-3,1,1,1)$ as shown below.
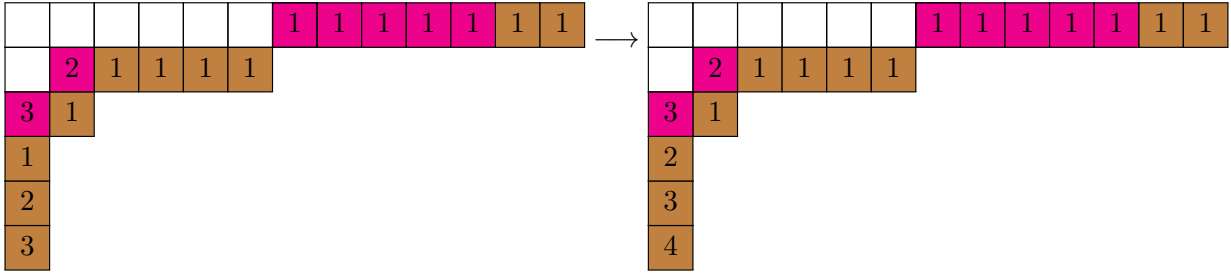
\begin{displaymath}
	\ydiagram{7,1}*[*(magenta) 1]{6+5,0,0}*[*(magenta) 2]{0,1+1}*[*(magenta) 3]{0,0,0+1}*[*(brown) 1]{11+2,2+4,1+1,0+1}*[*(brown) 2]{0,0,0,0,0+1}*[*(brown) 3]{0,0,0,0,0,0+1}
	\longrightarrow
	\ydiagram{7,1}*[*(magenta) 1]{6+5,0,0}*[*(magenta) 2]{0,1+1}*[*(magenta) 3]{0,0,0+1}*[*(brown) 1]{11+2,2+4,1+1,0}*[*(brown) 2]{0,0,0,0+1}*[*(brown) 3]{0,0,0,0,0+1}*[*(brown) 4]{0,0,0,0,0,0+1}
\end{displaymath}

\textbf{2. Suppose that the dominant partition $\beta$ contained in $\lambda/\alpha$ is equal to $\alpha=(m-1,1)$.}
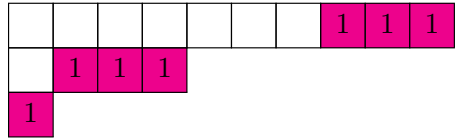
\begin{displaymath}
	\ydiagram{7,1}*[*(magenta) 1]{7+3,1+3,0+1}
\end{displaymath}
If the first column of $\lambda/\alpha$ contains the entry $2$, then we may replace $\beta$ with $(m-2,1,1)$ as shown below.
If the $2$ of $\beta$ appears in the first column of $\lambda$, then we may replace $\beta$ with $(m-2,1,1)$ as shown below.
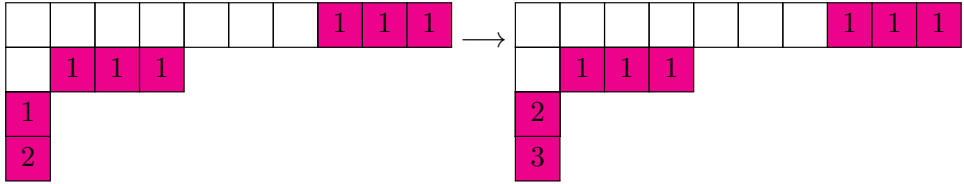
\begin{displaymath}
	\ydiagram{7,1}*[*(magenta) 1]{7+3,1+3,0+1}*[*(magenta) 2]{0,0,0,0+1}*[*(brown) 1]{0,0,0+1}
	\longrightarrow
	\ydiagram{7,1}*[*(magenta) 1]{7+3,1+3}*[*(magenta) 2]{0,0,0+1}*[*(magenta) 3]{0,0,0,0+1}*[*(brown) 1]{0,0,0+1}*[*(brown) 2]{0,0,0+1}
\end{displaymath}
Notice that $\lambda_3=1$.
Therefore, the dominant partition $\gamma$ contained in $\lambda/\alpha/\beta$ is simply equal to $(1^m)$.
Then we may replace $\beta$ with $(m-3,1,1,1)$  and $\gamma$ with $(2,1^{m-2})$ as shown below.
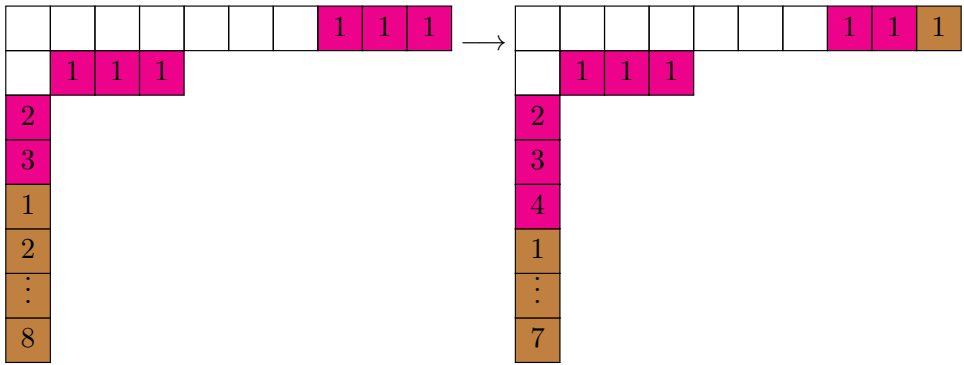
\begin{displaymath}
	\ydiagram{7,1}*[*(magenta) 1]{7+3,1+3}*[*(magenta) 2]{0,0,0+1}*[*(magenta) 3]{0,0,0,0+1}*[*(brown) 1]{0,0,0,0,0+1}*[*(brown) 2]{0,0,0,0,0,0+1}*[*(brown) \vdots]{0,0,0,0,0,0,0+1}*[*(brown) 8]{0,0,0,0,0,0,0,0+1}
	\longrightarrow
	\ydiagram{7,1}*[*(magenta) 1]{7+2,1+3}*[*(magenta) 2]{0,0,0+1}*[*(magenta) 3]{0,0,0,0+1}*[*(magenta) 4]{0,0,0,0,0+1}*[*(brown) 1]{9+1,0,0,0,0,0+1}*[*(brown) \vdots]{0,0,0,0,0,0,0+1}*[*(brown) 7]{0,0,0,0,0,0,0,0+1}
\end{displaymath}
If the second column of $\lambda/\alpha$ contains the entry $2$, then we may replace $\beta$ with $(m-2,1,1)$ as shown below.
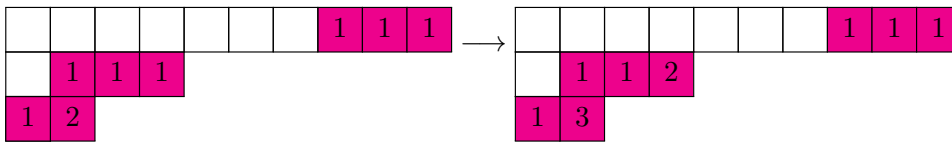
\begin{displaymath}
	\ydiagram{7,1}*[*(magenta) 1]{7+3,1+3,0+1}*[*(magenta) 2]{0,0,1+1}*[*(brown) 1]{0,0,0+1}
	\longrightarrow
	\ydiagram{7,1}*[*(magenta) 1]{7+3,1+2,0+1}*[*(magenta) 2]{0,3+1,0}*[*(magenta) 3]{0,0,1+1}
\end{displaymath}
Let $\gamma$ be the dominant partition of $m$ contained in $\lambda/\alpha/\beta$.
Since $\lambda/\alpha/\beta$ can have at most $m-2$ columns, therefore $\gamma$ is not equal to $(m)$ and $\alpha$.
Suppose that $\gamma=(1^m)$.
Then we may change the appearance of $\beta$ and replace $\gamma$ with $(2,1^{m-2})$ as shown below. 
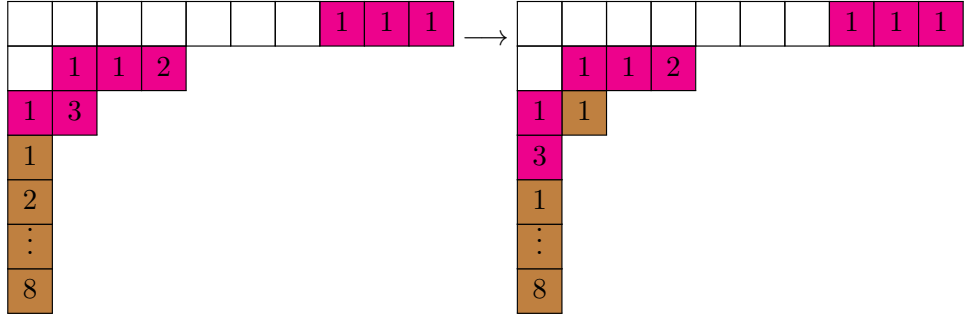
\begin{displaymath}
	\ydiagram{7,1}*[*(magenta) 1]{7+3,1+2,0+1}*[*(magenta) 2]{0,3+1,0}*[*(magenta) 3]{0,0,1+1}*[*(brown) 1]{0,0,0,0+1}*[*(brown) 2]{0,0,0,0,0+1}*[*(brown) \vdots]{0,0,0,0,0,0+1}*[*(brown) 8]{0,0,0,0,0,0,0+1}
	\longrightarrow
	\ydiagram{7,1}*[*(magenta) 1]{7+3,1+2,0+1}*[*(magenta) 2]{0,3+1,0}*[*(magenta) 3]{0,0,0,0+1}*[*(brown) 1]{0,0,1+1,0,0+1}*[*(brown) 2]{0,0,0,0,0}*[*(brown) \vdots]{0,0,0,0,0,0+1}*[*(brown) 8]{0,0,0,0,0,0,0+1}
\end{displaymath}
Suppose that $\gamma=\beta=(m-2,1,1)$.
Then we may replace $\gamma$ with $(m-3,1,1,1)$ as shown below.
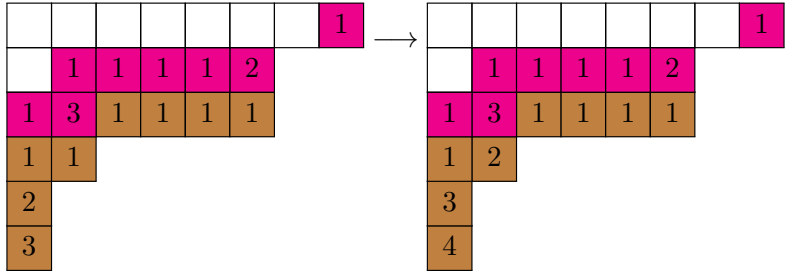
\begin{displaymath}
	\ydiagram{7,1}*[*(magenta) 1]{7+1,1+4,0+1}*[*(magenta) 2]{0,5+1,0}*[*(magenta) 3]{0,0,1+1}*[*(brown) 1]{0,0,2+4,0+2}*[*(brown) 2]{0,0,0,0,0+1}*[*(brown) 3]{0,0,0,0,0,0+1}*[*(brown) 4]{0,0,0,0,0,0}
	\longrightarrow
	\ydiagram{7,1}*[*(magenta) 1]{7+1,1+4,0+1}*[*(magenta) 2]{0,5+1,0}*[*(magenta) 3]{0,0,1+1}*[*(brown) 1]{0,0,2+4,0+1}*[*(brown) 2]{0,0,0,1+1}*[*(brown) 3]{0,0,0,0,0+1}*[*(brown) 4]{0,0,0,0,0,0+1}
\end{displaymath}
Finally, suppose that the $m$th column of $\lambda/\alpha$ contains the entry $2$.
Then we may replace $\beta$ with $(m-2,1,1)$ as shown below.
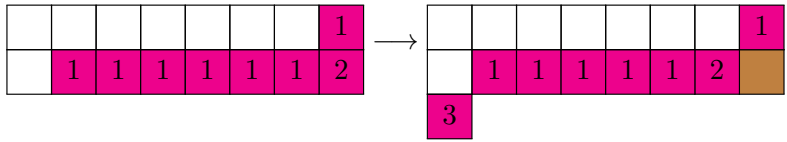
\begin{displaymath}
	\ydiagram{7,1}*[*(magenta) 1]{7+1,1+6,0}*[*(magenta) 2]{0,7+1,0}*[*(magenta) 3]{0,0}*[*(brown) 1]{0,0}
	\longrightarrow
	\ydiagram{7,1}*[*(magenta) 1]{7+1,1+5,0}*[*(magenta) 2]{0,6+1,0}*[*(magenta) 3]{0,0,0+1}*[*(brown) ]{0,7+1}
\end{displaymath}
Let $\gamma$ be the dominant partition contained in $\lambda/\alpha/\beta$.
If $\gamma\notin \{(m),(1^m),\alpha,\beta\}$, then we are done.
Otherwise, suppose that $\gamma=(m)$.
We may replace $\gamma$ with $(m-2,2)$ as shown below.
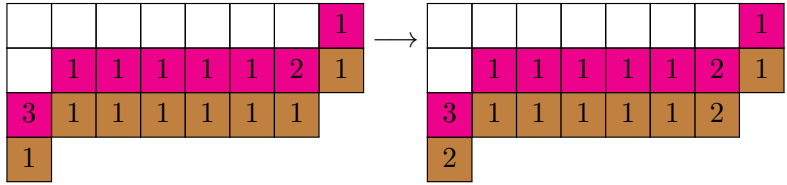
\begin{displaymath}
	\ydiagram{7,1}*[*(magenta) 1]{7+1,1+5,0}*[*(magenta) 2]{0,6+1,0}*[*(magenta) 3]{0,0,0+1}*[*(brown) 1]{0,7+1,1+6,0+1}
	\longrightarrow
	\ydiagram{7,1}*[*(magenta) 1]{7+1,1+5,0}*[*(magenta) 2]{0,6+1,0}*[*(magenta) 3]{0,0,0+1}*[*(brown) 1]{0,7+1,1+5,0}*[*(brown) 2]{0,0,6+1,0+1}
\end{displaymath}
Suppose that $\gamma=\alpha=(m-1,1)$.
If $\lambda_{m-1}'=3$, then $\lambda=(m,m,m)$.
In this case, there are many triples of partitions we can choose.
One of them is to replace $\alpha$ with $(m-2,2)$, $\beta$ with $(m-3,3)$ and $\gamma$ with $(m-2,1,1)$.

If $\lambda_2'=4$, we may replace $\gamma$ with $(m-2,2)$ as shown below.
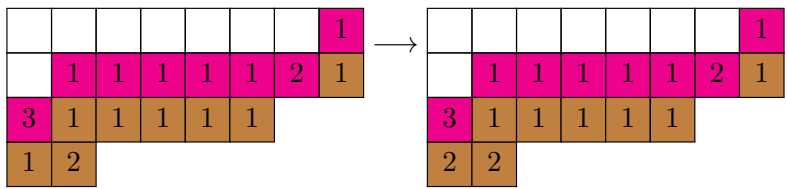
\begin{displaymath}
	\ydiagram{7,1}*[*(magenta) 1]{7+1,1+5,0}*[*(magenta) 2]{0,6+1,0}*[*(magenta) 3]{0,0,0+1}*[*(brown) 1]{0,7+1,1+5,0+1}*[*(brown) 2]{0,0,0,1+1,0}
	\longrightarrow
	\ydiagram{7,1}*[*(magenta) 1]{7+1,1+5,0}*[*(magenta) 2]{0,6+1,0}*[*(magenta) 3]{0,0,0+1}*[*(brown) 1]{0,7+1,1+5,0}*[*(brown) 2]{0,0,0,0+2,0}
\end{displaymath}
If $\lambda_1'=5$, then we may replace $\beta$ with $(m-2,2)$ and $\gamma$ with $(m-2,1,1)$ as shown below.
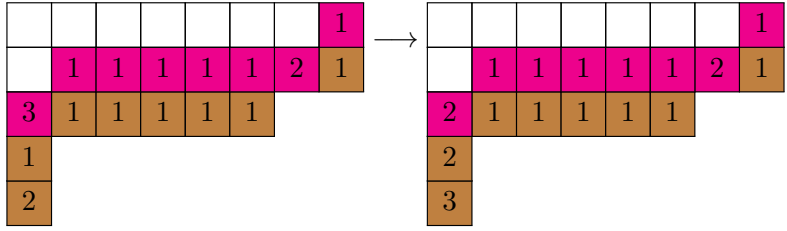
\begin{displaymath}
	\ydiagram{7,1}*[*(magenta) 1]{7+1,1+5,0}*[*(magenta) 2]{0,6+1,0}*[*(magenta) 3]{0,0,0+1}*[*(brown) 1]{0,7+1,1+5,0+1}*[*(brown) 2]{0,0,0,0,0+1,0}
	\longrightarrow
	\ydiagram{7,1}*[*(magenta) 1]{7+1,1+5,0}*[*(magenta) 2]{0,6+1,0+1}*[*(magenta) 3]{0,0,0}*[*(brown) 1]{0,7+1,1+5,0}*[*(brown) 2]{0,0,0,0+1,0}*[*(brown) 3]{0,0,0,0,0+1,0}
\end{displaymath}
Suppose that $\gamma=\beta=(m-2,1,1)$.
Then we may replace $\gamma$ with $(m-3,1,1,1)$ as shown below.
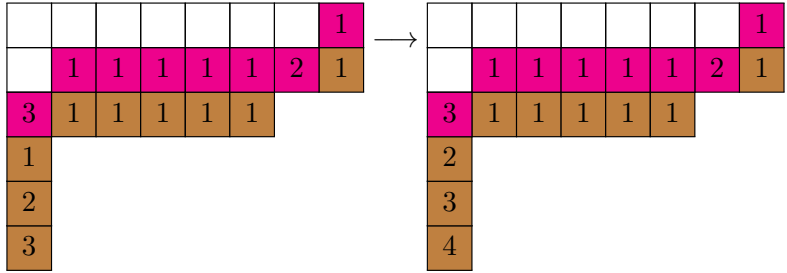
\begin{displaymath}
	\ydiagram{7,1}*[*(magenta) 1]{7+1,1+5,0}*[*(magenta) 2]{0,6+1,0}*[*(magenta) 3]{0,0,0+1}*[*(brown) 1]{0,7+1,1+5,0+1}*[*(brown) 2]{0,0,0,0,0+1,0}*[*(brown) 3]{0,0,0,0,0,0+1,0}
	\longrightarrow
	\ydiagram{7,1}*[*(magenta) 1]{7+1,1+5,0}*[*(magenta) 2]{0,6+1,0}*[*(magenta) 3]{0,0,0+1}*[*(brown) 1]{0,7+1,1+5}*[*(brown) 2]{0,0,0,0+1,0}*[*(brown) 3]{0,0,0,0,0+1,0}*[*(brown) 4]{0,0,0,0,0,0+1,0}
\end{displaymath}

Hence, the dominant partition $\beta$ contained in $\lambda/\alpha$ is different from $\alpha=(m-1,1)$ and $(m)$.
Equivalently, $\lambda/\alpha$ has at most $m-2$ columns.
Let $\mu$ be the dominant partition of $2m$ contained in $\lambda/\alpha$.
Then by LR-rule, we have $s_{\lambda/\alpha}\geq s_\mu$.
Using Lemma~\ref{lemma:base}, we have $s_\mu\geq s_\beta s_\gamma$ for some partitions $\beta$ and $\gamma$ with $\beta\neq\gamma$ unless $\mu$ is one of $(2m),(2m-1,1),(2m-2,1,1),(2m-2,2), (m,m),(2^m),(3,1^{2m-3}), (2^2,1^{2m-4}), \linebreak (2,1^{2m-2}),(1^{2m})$.
Note that such $\beta$ and $\gamma$ are not equal to $\alpha$ since $\max(\beta_1,\gamma_1)\leq \mu_1\leq m-2$.
Now let us look at the case where $\mu$ is one 
of the exceptions listed above.
Since $\mu_1\leq m-2$, therefore $\mu$ is not equal any of $(2m), (2m-1,1), (2m-2,1,1), (2m-2,2), (m,m)$.
We have assumed that $\lambda_1\geq \lambda_1'$, therefore $\lambda_1'\leq m+m/2$.
In particular, $\lambda_1'\leq 2m-5$.
This excludes all the cases except $\mu=(2^m)$.
If $\mu=(2^m)$, then this implies that $\lambda/\alpha$ has just two columns each with $m$ cells.
At the same time, we know that $\lambda/\alpha$ has a cell in the $m$th column.
This is a contradiction.

Hence, In all cases, we have constructed pairwise distinct partitions
$\alpha,\beta,\gamma \neq (m),(1^m)$ such that
$s_\lambda \geq s_\alpha s_\beta s_\gamma$.
This completes the proof.
\end{proof}

In fact, we can show the stronger version of the above lemma which will be useful in giving a lower bound on multiplicity.
\begin{lemma}
	\label{lemma:3_m_strongly_distinct}
	Let $m\geq 11$ be an integer.
	If $\lambda$ be a partition of $3m$ with $\lambda_1,\lambda_1'\leq 3m-6$, then there exist three mutually distinct partitions $\alpha,\beta,\gamma$ such that each of $\alpha,\beta,\gamma$ are different from $ (m),(1^m)$ and at least two of $\alpha,\beta,\gamma$ are different from $(m-1,1)$ and $(2,1^{m-2})$
	with $s_\alpha s_\beta s_\gamma\geq  s_\lambda$.
\end{lemma}
\begin{proof}
	Let $\lambda$ be a partition of $3m$ with $\lambda_1,\lambda_1'\leq 3m-6$.
	Using Lemma~\ref{lemma:3m_distinct}, there exist three mutually distinct partitions $\alpha,\beta,\gamma$ such that none of them is equal to any of $(m),(1^m)$ and $s_\alpha s_\beta s_\gamma\geq s_\lambda$.
	If two of them are different from $(m-1,1)$ and $(2,1^{m-2})$, then we are done.
	Otherwise, without loss of generality, we may assume that $\gamma$ is different from $(m-1,1)$ and $(2,1^{m-2})$ and $\alpha=(m-1,1)$, $\beta=(2,1^{m-2})$.
	Then we may replace $\alpha$ and $\beta$ with one of the following pairs of partitions:
	\begin{itemize}
		\item $\alpha=(m-2,1,1)$, $\beta=(3,1^{m-3})$ and $\gamma\neq \alpha,\beta$,
		\item $\alpha=(m-3,1,1,1)$, $\beta=(4,1^{m-4})$ and $\gamma\neq \alpha,\beta$.
	\end{itemize}
	This completes the proof.
\end{proof}

\begin{lemma}
	\label{lemma:3m+m+m}
	Let $m\geq 11$ be an integer.
	Then we have
	\begin{displaymath}
		\sum\limits_{\substack{\alpha\vdash 3m\\ \alpha_1,\alpha_1'\leq 3m-6}} s_\alpha 
		\sum\limits_{\substack{\beta \vdash m\\ \beta\neq (m-1,1), (2,1^{m-2})}} s_\beta
		\sum\limits_{\substack{\gamma \vdash m\\ \gamma\neq (m-1,1), (2,1^{m-2})}} s_\gamma
		\geq s_\lambda
	\end{displaymath}
	for all partitions $\lambda$ of $5m$ with $\lambda_1,\lambda_1'\leq 5m-6$.
\end{lemma}
\begin{proof}
	Using Lemma~\ref{lemma:p_p_except_standard_and_its_prime}, we have
	\begin{displaymath}
		\sum\limits_{\substack{\alpha\vdash 3m\\ \alpha_1,\alpha_1'\leq 3m-6}} s_\alpha 
		\sum\limits_{\substack{\beta \vdash m\\ \beta\neq (m-1,1), (2,1^{m-2})}} s_\beta
		\sum\limits_{\substack{\gamma \vdash m\\ \gamma\neq (m-1,1), (2,1^{m-2})}} s_\gamma
		\geq
		\sum\limits_{\substack{\alpha\vdash 3m\\ \alpha_1,\alpha_1'\leq 3m-6}} s_\alpha
		\sum\limits_{\substack{\nu \vdash 2m}} s_\nu.
	\end{displaymath}
	Let $\lambda$ be a partition of $5m$ with $\lambda_1,\lambda_1'\leq 5m-6$.
	Without loss of generality, we may assume that $\lambda_1\geq \lambda_1'$.
	Let $\tilde\alpha$ be a dominant partition of $6$ contained in $\tilde\lambda=(\lambda_2,\lambda_3,\ldots)$.
	Let $\alpha$ be the dominant partition of $3m$ contained in $\lambda$ such that the partition $(\alpha_2,\alpha_3,\ldots)$ contains $\tilde\alpha$.
	Then $\alpha_1,\alpha_1'\leq 3m-6$ by construction.
	Now let $\nu$ be a dominant partition of $2m$ contained in $\lambda/\alpha$.
	Then the Littlewood--Richardson rule implies $s_\alpha s_\nu \le s_\lambda$,
which gives the desired result.
\end{proof}
The following lemma is a repeated application of Lemma~\ref{lemma:3m+m+m}.
\begin{lemma}
	\label{lemma:3m_k_multiple_m}
	Let $m\geq 11$ be an integer and $t$ be an even positive integer.
	\begin{displaymath}
		\sum\limits_{\substack{\alpha\vdash 3m\\ \alpha_1,\alpha_1'\leq 3m-6}} s_\alpha 
		\left(
		\sum\limits_{\substack{\beta \vdash m\\ \beta\neq (m-1,1), (2,1^{m-2})}} s_\beta
		\right)^t
		\geq s_\lambda
	\end{displaymath}
	for all partitions $\lambda$ of $3m+tm$ with $\lambda_1,\lambda_1'\leq 3m+tm-6$.
\end{lemma}

\section{Proof of the main theorems}
In this section, we denote the one dimensional character of $D_n$ obtained by taking the $w_n$ to $1$ and $s$ to $-1$ by $-\mathbb{1}_{D_n}$.
Let us recall the following lemma which is one of the statements of \cite[Theorem 1.9]{Swanson}.
\begin{lemma}\label{lemma:dimension_bound}
	Let $n> 81$ be a positive integer and $\lambda$ be a partition of $n$ with $\lambda_1< n-7$ and $\lambda_1' <n-7$.
	Then $f^\lambda\geq n^5$.
\end{lemma}

Our first lemma is the validity of our main theorem when $n$ is prime.
\begin{lemma}
	\label{lemma:prime_real}
	Let $n$ be a prime number greater than $10$.
	Then Theorem~\ref{theorem:dihedral_branching_positivity} holds  for $n$.
	Furthermore, we have 
	\begin{align*}
		\langle \Res_{D_n}^{S_n} \chi_\lambda,\psi \rangle
		&\geq n^2,
	\end{align*}
	for all irreducible characters $\psi$ of $D_n$ and all partitions $\lambda$ of $n$ with $\lambda_1,\lambda_1'< n-7$.
\end{lemma}
\begin{proof}
	We shall verify the lemma for primes $n$ less than or equal to $81$ using \texttt{SageMath}.
	Therefore, we may assume that $n>81$.
	We can also verify the lemma for the partitions $\lambda$ with $\lambda_1\geq n-7$ or $\lambda_1'\geq n-7$ by direct computation.
	Therefore, let us consider partitions $\lambda$ of $n$ with $\lambda_1,\lambda_1'<n-7$.
	We shall start with the one-dimensional characters of $D_n$.
	Let us do the following computation.
	\begin{align*}
		\langle \Res_{D_n}^{S_n} \chi_\lambda,\pm \mathbb{1}_{D_n} \rangle
		&= \chi_\lambda(1)+(n-1) \chi_\lambda(w_n) \pm n \chi_\lambda(w_{(2^{(n-1)/2},1)})
	\end{align*}
	Let us assume that $\lambda$ is not a hook partition.
	Then $\chi_\lambda(w_n)=0$. Hence,
	\begin{align}\label{eq:res_D_n_pm}
		\langle \Res_{D_n}^{S_n} \chi_\lambda,\pm\mathbb{1}_{D_n} \rangle
		&= \tfrac{ \chi_\lambda(1)\pm n \chi_\lambda(w_{(2^{(n-1)/2,1})})}{2n}
	\end{align}

	Let us give an upper bound for $|\chi_\lambda(w_{(2^{(n-1)/2},1)})|$.
	Firstly, we have
	\begin{align*}
		\chi_\lambda(w_{(2^{(n-1)/2},1)})
		&= \sum\limits_{\eta\in \lambda^-} \chi_\eta(w_{(2^{(n-1)/2})}),
	\end{align*}
	where $\lambda^-$ is the set of partitions obtained by removing a cell from $\lambda$.
	If $|\lambda^-|=k$, then there are $k$ removable cells in the Young diagram of $\lambda$.
	Thus, $\lambda$ has at least $k$ many distinct parts.
	Therefore, $k(k+1)/2\leq n$.
	Hence, $k\leq \sqrt{2n-1}$.
	Thus, we have $|\lambda^-|\leq \sqrt{2n-1}$.

	If the $2$-core of $\eta$ is non-empty, then $\chi_\eta(w_{(2^{(n-1)/2})})=0$.
	Otherwise, using a recursive Murnaghan-Nakayama formula (see~\cite[2.7.32]{JamesKerber}, also see~\cite[Theorem 1.1]{Fomin_Lulov_1995}), we have
	\begin{align*}
		|\chi_\eta(w_{(2^m)})|
		&=\tfrac{m!2^m}{\sqrt{(2m)!}} \sqrt{f^\lambda}\\
		&= \frac{\sqrt{(2m)!}}{1~ 3~5~\dots~(2m-1)} \sqrt{f^\lambda}\\
		&= \frac{2~4~6~\dots~(2m)}{1~ 3~5~\dots~(2m-1)} \sqrt{f^\lambda}\\
		&= \frac{2~4~6~\dots~(2m-2)(2m)}{1~ 3~5~\dots~(2m-3)~(2m-1)~2m} \sqrt{2m} \sqrt{f^\lambda}\\
		&\leq \sqrt{2m} \sqrt{f^\lambda},\\
	\end{align*}
	where $m=(n-1)/2$.
	Hence, we have
	\begin{align*}
		|\chi_\lambda(w_{(2^{(n-1)/2},1)})|
		&= \sum\limits_{\eta\in \lambda^-} |\chi_\eta(w_{(2^{(n-1)/2})})|\\
		&\leq \sum\limits_{\eta\in \lambda^-} \sqrt{2m} \sqrt{f^\eta}\\
		&\leq \sqrt{2m} \sum\limits_{\eta\in \lambda^-} \sqrt{f^\lambda}\\
		&\leq \sqrt{2m} \sqrt{2n-1} \sqrt{f^\lambda}\\
		&= \sqrt{(n-1)(2n-1)} \sqrt{f^\lambda}\\
	\end{align*}
	Now we shall compute the numerator in the RHS of Equation~\eqref{eq:res_D_n_pm}.
	\begin{align*}
		|\chi_\lambda(1) \pm n \chi_\lambda(w_{(2^{(n-1)/2},1)})|
		& \geq \chi_\lambda(1) - n|\chi_\lambda(w_{(2^{(n-1)/2},1)})|\\
		& \geq f^\lambda - n\sqrt{(n-1)(2n-1)} \sqrt{f^\lambda}\\
		&= \sqrt{f^\lambda}(\sqrt{f^\lambda} - n\sqrt{(n-1)(2n-1)})
	\end{align*}
	Finally, 
	\begin{align*}
		\langle \Res_{D_n}^{S_n} \chi_\lambda,\pm\mathbb{1}_{D_n} \rangle
		&= \tfrac{ \chi_\lambda(1)\pm n \chi_\lambda(w_{(2^{(n-1)/2,1})})}{2n}\\
		&\geq \tfrac{\sqrt{f^\lambda}(\sqrt{f^\lambda} - n\sqrt{(n-1)(2n-1)})}{2n}
	\end{align*}
	Therefore, by lemma~\ref{lemma:dimension_bound}, we have $f^\lambda> n^5$.
	Then, we have
	\begin{align*}
		\tfrac{\sqrt{f^\lambda}(\sqrt{f^\lambda} - n\sqrt{(n-1)(2n-1)})}{2n}
		&> \tfrac{\sqrt{n^5}(\sqrt{n^5} - n\sqrt{(n-1)(2n-1)})}{2n}\\
		&> n^3\\
	\end{align*}
	Hence, we have $\langle \Res_{D_n}^{S_n} \chi_\lambda,\pm\mathbb{1}_{D_n} \rangle> n^3>\frac{n}{12}$.
	
	Let us now consider the case when $\lambda$ is a hook partition.
	Firstly, we have
	\begin{align*}
		\langle \Res_{D_n}^{S_n} \chi_{(n-k,1^k)},\pm\mathbb{1}_{D_n} \rangle
		&= (\chi(1) + (n-1) \chi(w_n) \pm  n  \chi(w_{(2^{(n-1)/2},1)}))/2n\\
		&\geq \left(f^\lambda + (n-1)(-1)^k -  n \tfrac{m! 2^m \sqrt{f^\lambda}(\sqrt{k}+\sqrt{n-k-1})}{\sqrt{(2m)!n}}\right)/2n\\
		&\geq \left(f^\lambda - (n-1) - n \tfrac{m! 2^m \sqrt{f^\lambda}(\sqrt{k}+\sqrt{n-k-1})}{\sqrt{(2m)!n}}\right)/2n\\
		&\geq \sqrt{f^\lambda} (\sqrt{f^\lambda} - n \tfrac{m! 2^m (\sqrt{k}+\sqrt{n-k-1})}{\sqrt{(2m)!n}})/2n\\
		&\geq \sqrt{f^\lambda} (\sqrt{f^\lambda} - n+1 -2n\sqrt{n})/2n\\
		&\geq n^{5/2} (n^{5/2} - n+1 -2n\sqrt{n})/2n\\
		&\geq n^3.
	\end{align*}

	Now let us consider the irreducible character $\delta$ of $D_n$ with degree $2$.
	Then for non-hook partitions $\lambda$, we have
	\begin{align*}
		\langle \Res_{D_n}^{S_n} \chi_\lambda,\delta \rangle
		&= \tfrac{f^\lambda}{n}>n^4.\\
	\end{align*}
	Now for the hook partitions $\lambda=(n-k,1^k)$, we have
	\begin{align*}
		\langle \Res_{D_n}^{S_n} \chi_\lambda,\delta \rangle
		&\geq \tfrac{2f^\lambda-2n}{2n}\\
		&= \tfrac{f^\lambda - n}{n} \geq n^4.
	\end{align*}
	This completes the proof.
\end{proof}

We have a simple corollary of the above lemma for the alternating group $A_n$.
\begin{corollary}
	\label{corollary:prime_alternating}
	Let $n$ be a prime number greater than $10$.
	Then Theorem~\ref{theorem:dihedral_branching_alternating} holds for $n$.
	Furthermore, for $n>20$, we have
	\begin{align*}
		\langle \Res_{A_n}^{S_n} \chi_\lambda,\psi \rangle
		&\geq n^2,
	\end{align*}
	for all irreducible characters $\psi$ of $A_n$ and all partitions $\lambda$ of $n$ with $\lambda_1,\lambda_1'< n-7$.
\end{corollary}


	

\subsection*{Proof of Theorem~\ref{theorem:dihedral_branching_positivity}.}
Let $n$ be a positive integer and $\lambda$ be a partition of $n$.
The proof is by induction on $n$.
Firstly, we shall verify these theorems using \texttt{SageMath} for $n\leq 81$.
So let us assume that $n > 81$.
We may start by assume that $\lambda_1,\lambda_1'< n-7$.

\subsubsection*{Case 1: $n$ is a prime.}
In this case the results follow from Lemma~\ref{lemma:prime_real}.

\subsubsection*{Case 2: $n$ is odd composite.}
Let $n=pm$ where $p$ is the smallest prime divisor of $n$.
Note that $p\leq \sqrt{n}$, so $m\geq \sqrt{n}$.
Thus, $m\geq 11$.
Let $\lambda_1,\lambda_1'< n-7$.
Using Lemma~\ref{lemma:3m_k_multiple_m}, we have partitions $\mu^1,\mu^2, \dotsc, \mu^p$ of $m$ such that none of them are equal to $(m-1,1)$ and $(2,1^{m-2})$ and at least three of them are mutually distinct with none of them equal to $(m)$ or $(1^m)$ such that $s_{\mu^1} s_{\mu^2} \dotsb s_{\mu^p}\leq s_\lambda$.

Using Proposition~\ref{proposition:Giannelli}, we have
\begin{align*}
	\Res_{D_n}^{S_n} \chi_\lambda \geq \Ind_{\langle w_n^p \rangle}^{D_n} \Res^{S_m^{\times p}}_{\langle w_n^p \rangle}
	\chi_{\mu^1}\times \chi_{\mu^2} \times \dotsb \times \chi_{\mu^p} 
\end{align*}
Let $\mu^r, \mu^s, \mu^t$ be three mutually distinct partitions among $\mu^1,\mu^2, \dotsc, \mu^p$.
Then by induction, we have
\begin{align} \label{eq:r_s_t_lowerbound}
	\Res^{S_m}_{C_m} \chi_{\mu^i} \geq \tfrac{m}{6} \zeta 
\end{align}
for $i=r,s,t$, and for all linear characters $\zeta$ of $C_m$.

Let us consider the two linear characters of $D_n$.
For this purpose, first let us choose the trivial character $\zeta_i$ ($= \mathbb{1}_{C_m}$) of $C_m$ for all $i$ other than $r,s,t$.
Then first we may choose any linear characters $\zeta_r$ and $\zeta_s$ of $C_m$ for $i=r$ and $i=s$ respectively.
Then we shall choose $\zeta_t$ for $i=t$ such that $\Res^{C_m\times C_m  \times C_m}_{\langle w_n^p \rangle} \zeta_r \times \zeta_s \times \zeta_t = \mathbb{1}_{\langle w_n^p \rangle}$.
Using Equation~\eqref{eq:r_s_t_lowerbound}, we have
\begin{align*}
	\Res^{S_m^{\times 3}}_{ \langle w_{3p}^3 \rangle} \chi_{\mu^r} \times \chi_{\mu^s} \times \chi_{\mu^t} &\geq {(\tfrac{m}{6})}^2 \zeta_r \times \zeta_s \times \zeta_t \geq \tfrac{m^2}{36} \mathbb{1}_{\langle w_{3p}^3 \rangle}
\end{align*}

Since there are $m$ linear characters of $C_m$, we may choose $\zeta_r$ and $\zeta_s$ in $m^2$ ways.
Thus, we have
\begin{align*}
	\Res^{S_m^{\times 3}}_{ \langle w_{3p}^3 \rangle} \chi_{\mu^r} \times \chi_{\mu^s} \times \chi_{\mu^t} &\geq m^2 \tfrac{m^2}{36} \mathbb{1}_{\langle w_{3p}^3 \rangle}.
\end{align*}

Hence, we have
\begin{align*}
	\Res_{D_n}^{S_n} \chi_\lambda &\geq \Ind_{\langle w_n^p \rangle}^{D_n} \Res^{S_m^{\times p}}_{\langle w_n^p \rangle}
	\chi_{\mu^1}\times \chi_{\mu^2} \times \dotsb \times \chi_{\mu^p} \\
	&\geq \tfrac{m^4}{36} \Ind_{\langle w_n^p \rangle}^{D_n} \mathbb{1}_{\langle w_n^p \rangle} \\
	&\geq \tfrac{m^4}{36} (\pm \mathbb{1}_{D_n} )\\
	&\geq n (\pm \mathbb{ 1}_{D_n}).
\end{align*}

Let us consider the non-linear characters $\psi$ of $D_n$.
Let $\delta$ be a linear character of $C_m$ such that $\Res^{D_n}_{C_m} \psi \geq \delta$.
Like before, we may choose $\zeta_i=\mathbb{1}_{C_m}$ for all $i$ other than $r,s,t$.
Then we may choose any linear characters $\zeta_r$ and $\zeta_s$ of $C_m$ for $i=r$ and $i=s$ respectively.
Then we shall choose $\zeta_t$ for $i=t$ such that $\Res^{C_m\times C_m  \times C_m}_{\langle w_{n}^p \rangle} \zeta_r \times \zeta_s \times \zeta_t (w_{3p}^3) = \delta(w_{3p}^3)$.
Using exactly the same calculations as before, we have
\begin{align*}
	\Res_{D_n}^{S_n} \chi_\lambda &\geq n \psi.
\end{align*}

If $\lambda_1$ or $\lambda_1'$ is greater than or equal to $n-7$, then we may check the theorem direct computation of characters.

\subsubsection*{Case 3: $n$ is even.}
Let us write $n=2m$.
Notice that $m\geq 41$.
Suppose that $\lambda_1,\lambda_1'< n-7$ and that $\lambda$ is not one of $(m,m), (2^m)$.
Then using Lemma~\ref{lemma:base}, we have distinct partitions $\alpha$ and $\beta$ of $m$ that are different from $(m), (m-1,1), (2,1^{m-2})$ and $(1^m)$ such that $s_\alpha s_\beta \geq s_\lambda$.
Now Proposition~\ref{proposition:even_Giannelli} gives us 
\begin{align*}
	\Res_{D_n}^{S_n} \chi_\lambda &\geq \Ind^{D_n}_{\langle w_n^2, s \rangle} \Res^{S_m^{\times 2}}_{\langle w_n^2, s \rangle} \chi_\alpha \times \chi_\beta,
\end{align*}
where $s$ is an involution of $D_n$.

Let us start with an irreducible character $\zeta$ of $D_n$.
Let $\delta$ be an irreducible character of $\langle w_n^2, s \rangle$ such that $\Res^{D_n}_{\langle w_n^2, s \rangle} \zeta \geq \delta$.
Notice that for any two-dimensional irreducible character $\psi$ of $D_m$, there exists a two-dimensional irreducible character $\tilde{\psi}$ of $D_m$ such that 
$\Res^{D_m \times D_m}_{\langle w_n^2, s \rangle} \psi \times \tilde{\psi} \geq \delta$.

By induction hypothesis, we have $\Res^{S_m}_{D_m} \chi_\alpha \geq \tfrac{m}{6} \psi$ and $\Res^{S_m}_{D_m} \chi_\beta \geq \tfrac{m}{6} \tilde{\psi}$ for all two-dimensional irreducible characters $\psi$ and $\tilde{\psi}$ of $D_m$, we have
\begin{align*}
	\Res^{S_m^{\times 2}}_{\langle w_n^2, s \rangle} \chi_\alpha \times \chi_\beta 
	&= \Res^{D_m \times D_m}_{\langle w_n^2, s \rangle} \Res^{S_m}_{D_m \times D_m} \chi_\alpha \times \chi_\beta \\
	&\geq {(\tfrac{m}{6})}^2 \Res^{D_m \times D_m}_{\langle w_n^2, s \rangle} \psi \times \tilde{\psi} \\
	&\geq \tfrac{m^2}{36} \delta.
\end{align*} 
Since there are at least $\tfrac{m-2}{2}$ irreducible characters of $D_m$ of degree $2$, 
we have
\begin{align*}
	\Res^{S_m^{\times 2}}_{\langle w_n^2, s \rangle} \chi_\alpha \times \chi_\beta 
	&\geq \tfrac{m-2}{2} {(\tfrac{m}{6})}^2 \delta \\
	&\geq n \delta.
\end{align*}

Thus, we have
\begin{align*}
	\Res_{D_n}^{S_n} \chi_\lambda &\geq \Ind^{D_n}_{\langle w_n^2, s \rangle} \Res^{S_m^{\times 2}}_{\langle w_n^2, s \rangle} \chi_\alpha \times \chi_\beta \\
	&\geq n \Ind^{D_n}_{\langle w_n^2, s \rangle} \delta \\
	&\geq n \zeta.
\end{align*}

Now let us assume that $\lambda = (m,m).$
If $m$ is prime, then the result holds from Proposition~\ref{proposition:two_row_rectangular}.
Otherwise, $n=4k$ where $k=\tfrac{m}{2}$ is an integer greater than $20$.
Let $u= \lceil{\tfrac{k}{2}} \rceil$.
Using Corollary~\ref{corollary:divisible_by_4} with $\mu^1=(u+3,m-u-3), \mu^2=(u+2, m-u-2), \mu^3=(u+1,m-u-1), \mu^4=(3u+6-m,2m-3u-6)$  and using induction, we have 
\begin{align*}
	\langle \Res_{D_n}^{S_n} \chi_\lambda, \psi \rangle &\geq n
\end{align*}
for all irreducible characters $\psi$ of $D_n$.

For $\lambda = (2^m)$, we can tensor $\chi_\lambda$ with the sign character and use the previous case to get the result.
If $\lambda_1$ or $\lambda_1'$ is greater than or equal to $n-7$, then we may check the theorem direct computation of characters.
This completes the proof of Theorem~\ref{theorem:dihedral_branching_positivity}.

The proof of Theorem~\ref{theorem:dihedral_branching_positivity} gives us the following corollary.
\begin{corollary}
	\label{corollary:dihedral_branching_positivity_lower_bound}
	Let $n$ be a positive integer and $\lambda$ be a partition of $n$ with $\lambda_1,\lambda_1'< n-7$.
	Then we have
	\begin{align*}
		\langle \Res_{D_n}^{S_n} \chi_\lambda,\psi \rangle
		&\geq n,
	\end{align*}
	for all irreducible characters $\psi$ of $D_n$.
\end{corollary}

We now prove a lemma which will be useful in the proof of Theorem~\ref{theorem:dihedral_branching_alternating}.
\begin{lemma}
	\label{lemma:alternating_self_con_hook}
	Let $n=mp$ with $m\geq 11$ and $p$ be odd positive integers.
	Let $\lambda$ be the self-conjugate hook partition of $n$, i.e., $\lambda=(k+1,1^{k})$ where $n=2k+1$.
	Set $\mu^1=(\tfrac{m+1}{2},1^{\tfrac{m-1}{2}}), \mu^2=\mu^4=\dotsb=\mu^{p-1}=(m-2,2)$ and $\mu^3=\mu^5=\dotsb=\mu^{p}=(3,1^{m-3})$.
	Then we have
	\begin{align*}
		s_{\mu^1} s_{\mu^2} \dotsb s_{\mu^p} &\leq s_\lambda.
	\end{align*}
\end{lemma}

\subsection*{Proof of Theorem~\ref{theorem:dihedral_branching_alternating}.}
We shall verify the theorem directly for $n\leq 81$ using \texttt{SageMath}.
So let us assume that $n>81$.
If $n$ is prime, then the theorem follows from Corollary~\ref{corollary:prime_alternating}.
Let us now assume that $n$ is composite.
Let $n=pm$ where $p$ is the smallest prime divisor of $n$.
Thus, $m\geq 11$.
Let $\lambda$ be a non-self-conjugate partition of $n$.
Then using Theorem~\ref{theorem:dihedral_branching_positivity}, we have
\begin{align*}
	\langle \chi_\lambda, \Ind_{D_n}^{A_n} \psi \rangle
\end{align*}
for all irreducible characters $\psi$ of $D_n$.
Therefore, we may assume that $\lambda$ is a self-conjugate partition of $n$.
If $\lambda\neq (k+1,1^k)$ where $n=2k+1$, then $\lambda_1,\lambda_1'< n-7$.
Using the proof of Theorem~\ref{theorem:dihedral_branching_positivity}, we have
\begin{align*}
	\langle \chi_\lambda, \Ind_{D_n}^{S_n} \psi \rangle &\geq n
\end{align*}
for all irreducible characters $\psi$ of $D_n$.
Since $\chi_\lambda^\pm(g)=\tfrac{1}{2}\chi_\lambda(g)$ for all $g\in D_n$, we have
\begin{align*}
	\langle \chi_\lambda^\pm, \Ind_{D_n}^{A_n} \psi \rangle &\geq \tfrac{n}{2}.
\end{align*}
Now let us assume that $\lambda=(k+1,1^k)$ where $n=2k+1$.
Using Lemma~\ref{lemma:alternating_self_con_hook}, we have partitions $\mu^1=(\tfrac{m+1}{2},1^{\tfrac{m-1}{2}}), \mu^2=\mu^4=\dotsb=\mu^{p-1}=(m-2,2)$ and $\mu^3=\mu^5=\dotsb=\mu^{p}=(3,1^{m-3})$ of $m$ such that $s_{\mu^1} s_{\mu^2} \dotsb s_{\mu^p} \leq s_\lambda$.
Now using Proposition~\ref{proposition:Giannelli} and Theorem~\ref{theorem:dihedral_branching_positivity}, we have
\begin{align*}
	\Res_{D_n}^{S_n} \chi_\lambda &\geq \Ind_{\langle w_n^p \rangle}^{D_n} \Res^{S_m^{\times p}}_{\langle w_n^p \rangle}
	\chi_{\mu^1}\times \chi_{\mu^2} \times \dotsb \times \chi_{\mu^p}.
\end{align*}
Using Theorem~\ref{theorem:dihedral_branching_positivity}, we have that $\Res^{S_m}_{C_m} \chi_{\mu^i} \geq \tfrac{m}{6} \zeta$ for all linear characters $\zeta$ of $C_m$ for $i=1,2,\dotsc,p$.
Now using the same calculations as in the proof of Theorem~\ref{theorem:dihedral_branching_positivity}, we have 
\begin{align*}
	\Res_{D_n}^{S_n} \chi_\lambda \geq (\tfrac{m}{6})^p m^{\tfrac{p+1}{2}} \psi,
\end{align*}
for all irreducible characters $\psi$ of $D_n$.
Since $m\geq 11$ and $p\geq 3$, we have
\begin{align*}
	(\tfrac{m}{6})^3 m^{2} &\geq m^3\\
	(\tfrac{m}{6})^2 m &\geq 2m
\end{align*}
Thus, we have
\begin{align*}
	\Res_{D_n}^{S_n} \chi_\lambda &\geq n^2 \psi.
\end{align*}
Therefore, we have
\begin{align*}
	\langle \chi_\lambda^\pm, \Ind_{D_n}^{A_n} \psi \rangle 
	&\geq  \tfrac{1}{2} \langle \chi_\lambda, \Ind_{D_n}^{S_n} \psi \rangle - \tfrac{1}{4n} \sqrt{n}\phi(n)\\
	&\geq \tfrac{n^2}{2} - \tfrac{\sqrt{n}}{4}\\
	&> n.
\end{align*}
This completes the proof of Theorem~\ref{theorem:dihedral_branching_alternating}.

The proof of Theorem~\ref{theorem:dihedral_branching_alternating} gives us the following corollary.
\begin{corollary}
	\label{corollary:dihedral_branching_alternating_lower_bound}
	Let $n$ be a positive integer and $\lambda$ be a partition of $n$ with $\lambda_1,\lambda_1'< n-7$.
	Then we have
	\begin{align*}
		\langle \chi_\lambda, \Ind_{D_n}^{A_n} \psi \rangle
		&\geq \tfrac{n}{2},
	\end{align*}
	for all irreducible characters $\psi$ of $D_n$.
\end{corollary}

\subsection*{Acknowledgements}
I would like to thank Arvind Ayyer, V. Sathish Kumar, Amritanshu Prasad and Sheila Sundaram for their fruitful discussions, support and encouragement.
I am grateful to Alexey Staroletov for his comments on an earlier version of this manuscript.
I thank Maruthu Pandiyan B for providing useful references copies.
The author is supported by ANRF National Postdoctoral Fellowship (PDF/2025/003158).

\bibliographystyle{abbrv}
\bibliography{refs}
\end{document}